\documentclass[11pt,reqno]{amsart}          
\usepackage{geometry}                           
\usepackage{amsmath}
\usepackage{amsfonts}
\usepackage{graphicx}
\usepackage[onehalfspacing]{setspace}
\usepackage{hyperref}
\usepackage[ngerman,english]{babel}
\usepackage{datetime}
\usepackage{calligra}
\usepackage{amssymb}
\usepackage{geometry}
\usepackage{cite}
\usepackage{amsthm}
\usepackage{url}
\usepackage{mathtools}
\usepackage{mathrsfs}
\usepackage{datetime}
\usepackage{enumitem}
\usepackage{accents}
\usepackage{xcolor}

\hypersetup{
colorlinks=true,
linkcolor=blue,
anchorcolor=blue,
citecolor=blue
}

\allowdisplaybreaks[1]

\def\Xint#1{\mathchoice
{\XXint\displaystyle\textstyle{#1}}%
{\XXint\textstyle\scriptstyle{#1}}%
{\XXint\scriptstyle\scriptscriptstyle{#1}}%
{\XXint\scriptscriptstyle\scriptscriptstyle{#1}}%
\!\int}
\def\XXint#1#2#3{{\setbox0=\hbox{$#1{#2#3}{\int}$ }
\vcenter{\hbox{$#2#3$ }}\kern-.6\wd0}}

\def\dashint{\Xint-}

\newtheorem{thm}{Theorem}[section]
\numberwithin{equation}{section}

\newtheorem{lemma}[thm]{Lemma}
\newtheorem*{defin}{Definition}

\newtheorem*{example}{Example}
\newtheorem{cor}[thm]{Corollary}
\newtheorem{prop}[thm]{Proposition}
\newtheorem*{rem}{Remark}

\makeatletter
\@namedef{subjclassname@2020}{\textup{2020} Mathematics Subject Classification}
\makeatother                                                                

\date{}
\title[On-diagonal heat kernel bounds]{Volume growth and on-diagonal heat kernel bounds on Riemannian manifolds with an end}
\author[A. Grigor'yan]{Alexander Grigor'yan}
\address[Alexander Grigor'yan]{Universit\"at Bielefeld, Fakult\"at f\"ur Mathematik, Postfach 100131, D-33501 Bielefeld, Germany}
\email{grigor@math.uni-bielefeld.de}
\author[P. S\"urig]{Philipp S\"urig}
\address[Philipp S\"urig]{Universit\"at Bielefeld, Fakult\"at f\"ur Mathematik, Postfach 100131, D-33501 Bielefeld, Germany}
\email{philipp.suerig@uni-bielefeld.de}

\begin{document}
\maketitle

\begin{abstract}
We investigate heat kernel estimates of the form $p_{t}(x, x)\geq c_{x}t^{-\alpha},$ for large enough $t$, where $\alpha$ and $c_{x}$ are positive reals and $c_{x}$ may depend on $x$, on manifolds having at least one end with a polynomial volume growth.
\end{abstract}
\pagestyle{headings}

\tableofcontents

\section{Introduction}

\let\thefootnote\relax\footnotetext{Funded by the Deutsche Forschungsgemeinschaft (DFG, German Research Foundation) – SFB 1283/2 2021 – 317210226}

Let $M$ be a complete connected non-compact Riemannian manifold and $p_{t}(x,y) $ be the \emph{heat kernel} on $M$, that is, the minimal positive fundamental solution of the heat equation $\partial _{t}u=\Delta u$, where $
\Delta $ is the Laplace-Beltrami operator on $M$.
In this paper, we investigate the long time behaviour of $p_{t}(x, x)$ for $t\to +\infty$, $x\in M$. Especially, we are interested in lower bounds for large enough $t$ of the form \begin{equation}\label{polydecayinintr}p_{t}(x, x)\geq c_{x}t^{-\alpha},\end{equation} where $\alpha$ and $c_{x}$ are positive reals and $c_{x}$ may depend on $x$.

Let $V(x, r)=\mu(B(x, r))$ be the volume function of $M$ where $B(x, r)$ denotes the geodesic balls in $M$ and $\mu$ the Riemannian measure on $M$. It was proved by A. Grigor'yan and T. Coulhon in \cite{Coulhon1997}, that if for some $x_{0}\in M$ and all large enough $r$, \begin{equation}\label{polynomialVGinintro}V(x_{0}, r)\leq Cr^{N}\end{equation} where $C$ and $N$ are positive constants, then \begin{equation}\label{heatkerneldecayloginintro}p_{t}(x, x)\geq \frac{c_{x}}{(t\log t)^{N/2}},\end{equation} which obviously implies (\ref{polydecayinintr}).

It is rather surprising that such a weak hypothesis as $(\ref{polynomialVGinintro})$ implies a pointwise lower bound (\ref{heatkerneldecayloginintro}) of the heat kernel. In this paper we obtain heat kernel bounds assuming even weaker hypotheses about $M$. We say that an open connected proper subset $\Omega$ of $M$ is an \emph{end} of $M$ if $\partial\Omega$ is compact but $\overline{\Omega}$ is non-compact (see also Section \ref{On-diagonal heat kernel lower bounds}).
One of our aims here is to obtain lower bounds for the heat kernel assuming only hypotheses about the intrinsic geometry of $\Omega$, although a priori it was not obvious at all that such results can exist.

One of motivations was the following question asked by A. Boulanger in \cite{Boulanger2020CountingPO} (although for a more restricted class of manifolds). Considering the volume function in $\Omega$ given by $$V_{\Omega}(x, r)=\mu(B(x, r)\cap\Omega),$$ Boulanger asked if the heat kernel satisfies (\ref{polydecayinintr}) provided it is known that \begin{equation}\label{polynomialVOmGinintro}V_{\Omega}(x_{0}, r)\leq Cr^{N},\end{equation} for some $x_{0}\in \Omega $ and all $r$ large enough.

A first partial answer to this question was given by A. Grigor'yan, who showed in \cite{Grigoryan2021}, that if (\ref{polynomialVOmGinintro}) holds and $\overline{\Omega}$, considered as a manifold with boundary, is \textit{non-parabolic}, (and hence, $N>2$ in (\ref{polynomialVOmGinintro})) then (\ref{heatkerneldecayloginintro}) is satisfied. More precisely, denoting by $p_{t}^{\Omega}(x, y)$ the heat kernel in $\Omega$ with the Dirichlet boundary condition on $\partial \Omega$, it was proved in \cite{Grigoryan2021} that, for all $x\in \Omega$ and large enough $t$, \begin{equation}\label{dirichlethatkernelinintro}p_{t}^{\Omega}(x, x)\geq \frac{c_{x}}{(t\log t)^{N/2}},\end{equation} which implies (\ref{heatkerneldecayloginintro}) by the comparison principle. 

From a probabilistic point of view, the estimate (\ref{dirichlethatkernelinintro}) for non-parabolic $\overline{\Omega}$ is very natural if one compares it with (\ref{heatkerneldecayloginintro}), since the non-parabolicity of $\overline{\Omega}$ implies that the probability that Brownian motion started in $\Omega$ never hits the boundary $\partial \Omega$ is positive (see [\cite{Grigorextquotesingleyan1999}, Corollary 4.6]). Hence, one expects that the heat kernel in $\overline{\Omega}$ and the heat kernel in $\Omega$ with Dirichlet boundary condition are comparable. 

The main direction of research in this paper is the validity of the estimate (\ref{polydecayinintr}) in the case when $\overline{\Omega}$ is parabolic and the volume function of $\Omega $ satisfies (\ref{polynomialVOmGinintro}). We prove (\ref{polydecayinintr}) for a certain class of manifolds $M$ when $\overline{\Omega}$ is parabolic as well as construct a class of manifolds $M$ with parabolic ends where (\ref{polydecayinintr}) does not hold. 

In Section \ref{On-diagonal heat kernel lower bounds} we are concerned with positive results. One of our main results - \textbf{Theorem \ref{thmlocharsph}}, ensures the estimate (\ref{polydecayinintr}) when $\overline{\Omega}$ is a \emph{locally Harnack} manifold (see Subsection \ref{positiveresultslower} for the definition). In order to handle difficulties that come from the parabolicity of the end, we use the method of $h$-\emph{transform} (see Subsection \ref{subsechtrans}). For that we construct a positive harmonic function $h$ in $\Omega$ and define a new measure $\widetilde{\mu}$ by $d\widetilde{\mu}=h^{2}d\mu$. Thus, we obtain a \emph{weighted manifold} $\left(\overline{\Omega},\widetilde{\mu}\right)$. We prove that this manifold is non-parabolic, satisfies the polynomial volume growth and, hence, the heat kernel $\widetilde{p}_{t}^{\Omega}$ of $\left(\Omega,\widetilde{\mu}\right)$ satisfies the lower bound (\ref{dirichlethatkernelinintro}). Then a similar lower bound for $p_{t}^{\Omega}$ and, hence, for $p_{t}$, follows from the identity
$$p_{t}^{\Omega}(x, x)=h^{2}(x)\widetilde{p}^{\Omega}_{t}(x, x)$$ (see Lemma \ref{relheattildeohne}). Note that the techniques of $h$-transform for obtaining heat kernel bounds was used in \cite{grigor2009heat} and \cite{Grigorextquotesingleyan2001} although in different settings.

In Section \ref{On-diagonal heat kernel upper bounds} we construct examples of manifolds $M$ having a parabolic end $\Omega$ with finite volume (in particular, satisfying (\ref{polynomialVOmGinintro})) but such that the heat kernel $p_{t}\left( x,x\right)$ decays \emph{superpolynomially} as $t\rightarrow \infty.$ In fact, the end $\Omega$ is constructed as a \textit{model manifold} (see Section \ref{secisobdry} for the definition of this term) that topologically coincides with $\left( 0,+\infty \right) \times \mathbb{S}^{n-1}$, $n\geq 2$, while the Riemannian metric on $\Omega$ is given by \begin{equation}\label{metricinintrowarp}ds^{2}=dr^{2}+\psi^{2}(r)d\theta^{2},\end{equation} where $d\theta ^{2}$ is a standard Riemannian metric on $\mathbb{S}^{n-1}$ and \begin{equation}\label{defvonpisinintro}\psi(r)=e^{-\frac{1}{n-1}r^{\alpha}},\end{equation}
with $0<\alpha \leq 1.$ Our second main result -\textbf{Theorem \ref{heatkernelforsmallendviah}}, says that for a certain manifold $M$ with this end $\Omega $ the following heat kernel estimate
holds: \begin{equation}\label{introheatkernelupper}p_{t}(x, x)\leq C_{x}\exp\left(-Ct^{\frac{\alpha}{2-\alpha}}\right),\end{equation} for all $x\in M$ and large enough $t$.
The estimate $(\ref{introheatkernelupper})$ follows from \textbf{Theorem \ref{heatonsumweightfaber}} where we obtain the upper bound of the heat kernel $\widetilde{p}_{t}$ of a weighted manifold $\left( M,\widetilde{\mu }\right) $ after an appropriate $h$-transform. In this theorem we prove that 
\begin{equation}\label{weightedheatupperinintro}\widetilde{p}_{t}(x, x)\leq C\exp\left(-C_{1}t^{\frac{\alpha}{2-\alpha}}\right).\end{equation} In fact, this decay is sharp, meaning that we have a matching lower bound $$\sup_{x\in M}\widetilde{p}_{t}(x, x)\geq c\exp\left(-C_{2}t^{\frac{\alpha}{2-\alpha}}\right)$$ (see the remark after Theorem \ref{heatonsumweightfaber}).

The key ingredient in the proof of \textbf{Theorem \ref{heatonsumweightfaber}} is obtaining a \textit{lower isoperimetric} function $J$ on $(\overline{\Omega}, \widetilde{\mu})$, which yields then the heat kernel upper bound (\ref{weightedheatupperinintro}) by a well-known technique (see [\cite{Grigoryan1999}, Proposition 7.1] and Proposition \ref{thmfaberheatupper}).
We say that a function $J$ on $[0, +\infty)$ is a lower isoperimetric function for $(\overline{\Omega}, \widetilde{\mu})$ if, for any precompact open set $U\subset \overline{\Omega}$ with smooth boundary, \begin{equation}\label{isoperimetricdefJintro}\widetilde{\mu}^{+}(U)\geq J(\widetilde{\mu}(U)),\end{equation} where $\widetilde{\mu}^{+}$ denotes the perimeter with respect to the measure $\widetilde{\mu}$ (see Section \ref{Isoperimetric inequalities for warped products} for more details).

In Section \ref{Isoperimetric inequalities for warped products} we present a technique for obtaining isoperimetric inequalities on \emph{warped products} of weighted manifolds. The isoperimetric inequality on \emph{Riemannian products} was proved in \cite{grigor1985isoperimetric}. We develop further the method of \cite{grigor1985isoperimetric} to deal with warped products, in particular, with the metric (\ref{metricinintrowarp}). The main result here is stated in \textbf{Theorem \ref{thm1iso}}.

Given two weighted manifolds $(M_{1}, \mu_{1})$ and $(M_{2}, \mu_{2})$ consider the weighted manifold $(M, \mu)$ such that $M=M_{1}\times M_{2}$ as topological spaces, the Riemannian metric $ds^{2}$ on $M$ is defined by $$ds^{2}=dx^{2}+\psi^{2}(x)dy^{2},$$ with $\psi$ being a smooth positive function on $M_{1}$ and $dx^{2}$ and $dy^{2}$ denoting the Riemannian metrics on $M_{1}$ and $M_{2}$, respectively and measure $\mu$ on $M$ is defined by $\mu=\mu_{1}\times \mu_{2}$. Assume that the function $\psi$ is bounded and $(M_{1}, \mu_{1})$ and $(M_{2}, \mu_{2})$ admit continuous lower isoperimetric functions $J_{1}$ and $J_{2}$, respectively. Then we prove in \textbf{Theorem \ref{thm1iso}}, that $(M, \mu)$ admits a lower isoperimetric function $$J(v)=c\inf_{\varphi, \phi}\left(\int_{0}^{\infty} {J_{1}(\varphi(t))dt}+\int_{0}^{\infty}{J_{2}(\phi(s))ds}.\right),$$ for some positive constant $c>0$ and where $\varphi$ and $\phi$ are \textit{generalized mutually inverse} functions such that $$v=\int_{0}^{\infty} {\varphi(t) dt}=\int_{0}^{\infty}{\phi(s)ds}.$$

As a consequence of \textbf{Theorem \ref{thm1iso}}, we obtain in Theorem \ref{propisoformod} that the
aforementioned weighted model manifold $\left(\overline{\Omega}, \widetilde{\mu}\right)$ admits a lower isoperimetric function $J$ such that for large enough $v$, 
$$J(v)=\frac{cv}{(\log v)^{\frac{2-2\alpha}{\alpha}}},$$ 
for some positive constant $c>0$, which leads to the estimate (\ref{weightedheatupperinintro}).


Even though we managed to give both positive and negative results for manifolds with parabolic end concerning the estimate (\ref{polydecayinintr}), a gap still remains. Closing this gap seems to be interesting for future work, for example, it might be desirable to construct a manifold with parabolic end of infinite volume for which (\ref{polydecayinintr}) does not hold. 

NOTATION. For any nonnegative functions $f, g$, we write $f\simeq g$ if there exists a constant $C>1$ such that $$C^{-1}f\leq g\leq Cf.$$

\section{On-diagonal heat kernel lower bounds}\label{On-diagonal heat kernel lower bounds}

Let $M$ be a non-compact Riemannian manifold with boundary $\delta M$ (which may be empty). Given a smooth positive function $\omega$ on $M$, let $\mu$ be the measure defined by 
\begin{equation*}\label{defweightmeasureome}d\mu=\omega^{2} d\textnormal{vol},\end{equation*} where $d\textnormal{vol}$ denotes the Riemannian measure on $M$. Similarly. we define $\mu'$ as the measure with density $\omega^{2}$ with respect to the Riemannian measure of codimension 1 on any smooth hypersurface. The pair ($M, \mu$) is called \textit{weighted manifold}.

The Riemannian metric induces the Riemannian distance $d(x, y),~ x, y\in M$.
Let $B(x, r)$ denote the geodesic ball of radius $r$ centered at $x$, that is $$B(x, r)=\{x\in M: d(x, y)<r\}$$ and $V(x, r)$ its volume on ($M, \mu$) given by $$V(x, r)=\mu(B(x, r)).$$
We say that $M$ is complete if the metric space ($M, d$) is complete. It is known that $M$ is complete, if and only if, all balls $B(x, r)$ are precompact sets. In this case, $V(x, r)$ is finite.

The Laplace operator $\Delta_{\mu}$ is the second order differential operator defined by 
$$\Delta_{\mu}f=div_{\mu}(\nabla f)=\omega^{-2}div(\omega^{2}\nabla f).$$ 
If $\omega\equiv 1$, then $\Delta_{\mu}$ coincides with the Laplace-Beltrami operator 
$\Delta=div \circ \nabla$.

Consider the Dirichlet form $$\mathcal{E}(u, v)=\int_{M}{(\nabla u, \nabla v)d\mu},$$ defined on the space $C_{0}^{\infty}(M)$ of smooth functions with compact support. The form $\mathcal{E}$ is closable in $L^{2}(M, \mu)$ and positive definit. Let us from now on denote by $\Delta_{\mu}$ its infinitisemal generator. By integration by parts, we obtain for all $u, v \in  C_{0}^{\infty}(M)$, \begin{equation}\label{intbypartsdelta}\mathcal{E}(u, v)=\int_{M}{(\nabla u, \nabla v)d\mu}=-\int_{M}{v\Delta ud\mu}-\int_{\delta M}{v\frac{\partial u}{\partial \nu}d\mu'},\end{equation} where $\nu$ denotes the inward unit normal vector field on $\delta M$. 
A function $u$ is called \textit{harmonic} if $u\in C^{2}(M)$, $\Delta u=0$ in $M\setminus \delta M$ and $\frac{\partial u}{\partial \nu}=0$ on $\delta M$.

The operator $\Delta_{\mu}$ generates the heat semi-group $P_{t}:=e^{t\Delta_{\mu}}$ which possesses a positive smooth, symmetric kernel $p_{t}(x, y)$.

Let $\Omega$ be an open subset of $M$ and denote $\delta \Omega:=\delta M\cap \Omega$. Then we can consider $\Omega$ as a manifold with boundary $\delta \Omega$. Hence, using the same constructions as above for $\Omega$ instead of $M$, we obtain the heat semigroup $P_{t}^{\Omega}$ with the heat kernel $p_{t}^{\Omega}(x, y)$,  which satisfies the Dirichlet boundary condition on $\partial \Omega$ and the Neumann boundary condition on $\delta \Omega$.

\begin{defin}\normalfont Let $M$ be a complete non-compact manifold. Then we call $\Omega$ an \textit{end} of $M$, if $\Omega$ is an open connected proper subset of $M$ such that $\overline{\Omega}$ is non-compact but $\partial \Omega$ is compact (in particular, when $\partial\Omega$ is a smooth closed hypersurface).
\end{defin}

In many cases, the end $\Omega$ can be considered as an exterior of a compact set of another manifold $M_{0}$, that means, $\Omega$ is isometric to $M_{0}\setminus K_{0}$ for some compact set $K_{0}\subset M_{0}$. If $(M, \mu)$ and $(M_{0}, \mu_{0})$ are weighted manifolds, with $\omega^{2}$ being the smooth density of measure $\mu$ and the measure $\mu_{0}$ having smooth density $\omega_{0}^{2}$, the isometry is meant in the sense of weighted manifolds, that is, this isometry maps measure $\mu$ to $\mu_{0}$ so that $\omega_{0}=\omega$ on $\Omega$.

A function $u\in C^{2}(M)$ is called \textit{superharmonic} if $\Delta_{\mu} u\leq 0$ in $M\setminus \delta M$ and $\frac{\partial u}{\partial \nu}\geq 0$ on $\delta M$, where $\nu$ is the outward normal unit vector field on $\delta M$. A \textit{subharmonic} function $u\in C^{2}(M)$ satisfies the opposite inequalities.
\begin{defin}\normalfont We say that a weighted manifold $(M, \mu)$ is \textit{parabolic} if any positive superharmonic function on $M$ is constant, and \textit{non-parabolic} otherwise.
\end{defin}

\begin{defin}\normalfont
Let $(M, \mu)$ be a weighted manifold and $\Omega$ be a subset of $M$. Then we define \textit{the volume function} of $\Omega$, for all $x\in M$ and $r>0$, by $$V_{\Omega}(x, r)=\mu(B_{\Omega}(x, r)),$$ where $B_{\Omega}(x, r)=B(x, r)\cap \Omega$.
\end{defin}

\begin{defin}\normalfont
Let $(M, \mu)$ be a weighted manifold. We say that $\Omega\subset M$ satisfies the \textit{polynomial volume growth condition}, if there exist $x_{0}\in \Omega$ and $r_{0}>0$ such that for all $r\geq r_{0}$, \begin{equation}\label{PolygrowthOM}V_{\Omega}(x_{0}, r)\leq Cr^{N},\end{equation} where $N$ and $C$ are positive constants.
\end{defin}

\begin{thm}[\cite{Grigoryan2021}, Theorem 8.3]\label{thmlowerptbd}
Let $M$ be a complete non-compact manifold with end $\Omega$. Assume that $\left(\overline{\Omega}, \mu\right)$ is a weighted manifold such that
\begin{itemize}
        \item $\left(\overline{\Omega}, \mu\right)$ is non-parabolic as a manifold with boundary $\partial \Omega\cup \delta \Omega$.
        \item $\Omega$ satisfies the polynomial volume growth condition (\ref{PolygrowthOM}) with $N>2$.
\end{itemize} Then for any $x\in \Omega$ there exist $c_{x}>0$ and $t_{x}>0$ such that for all $t\geq t_{x}$, \begin{equation}\label{lowerbndend}p_{t}^{\Omega}(x, x)\geq \frac{c_{x}}{(t\log t)^{N/2}},\end{equation} where $c_{x}$ and $t_{x}$ depend on $x$.
        
Consequently, if $(M, \mu)$ is a complete non-compact weighted manifold with end $\Omega$ such that the above conditions are satisfied, we have for any $x\in M$ and all $t\geq t_{x}$, \begin{equation}\label{nonparaendpoldecglob}p_{t}(x, x)\geq \frac{c_{x}}{(t\log t)^{N/2}}.\end{equation}
\end{thm}

\subsection{$h$-transform}\label{subsechtrans}

Recall that any smooth positive function $h$ induces a new weighted manifold $(M, \widetilde{\mu})$, where the measure $\widetilde{\mu}$ is defined by \begin{equation}\label{defvonmeasuremith}d\widetilde{\mu}=h^{2}d\mu\end{equation}
and we denote, for all $r>0$ and $x\in M$, by $\widetilde{V}(x, r)$ the volume function of measure $\widetilde{\mu}$.
The Laplace operator $\Delta_{\widetilde{\mu}}$ on $(M, \widetilde{\mu})$ is then given by $$\Delta_{\widetilde{\mu}}f=h^{-2}div_{\mu}(h^{2}\nabla f)=(h\omega)^{-2}div((h\omega)^{2}\nabla f).$$

\begin{lemma}[\cite{Grigorextquotesingleyan2001}, Lemma 4.1]\label{neuesmasdelta}
Assume that $\Omega\subset M$ is open and $\Delta_{\mu} h=0$ in $\Omega$. Then for any smooth function $f$ in $\Omega$, we have \begin{equation}\label{doobtrnsfmit}\Delta_{\widetilde{\mu}}f=h^{-1}\Delta_{\mu}(hf).\end{equation}
\end{lemma}

\begin{lemma}[\cite{Grigorextquotesingleyan2001}, Lemma 4.2]\label{relheattildeohne}
Assume that $h$ is a harmonic function in an open set $\Omega\subset M$. Then the Dirichlet heat kernels $p^{\Omega}_{t}$ and $\widetilde{p}^{\Omega}_{t}$ in $\Omega$, associated with the corresponding Laplace operators $\Delta_{\mu}$ and $\Delta_{\widetilde{\mu}}$, are related by \begin{equation}\label{relhtransformheat} p^{\Omega}_{t}(x, y)=h(x)h(y)\widetilde{p}^{\Omega}_{t}(x, y),\end{equation} for all $t>0$ and $x, y\in \Omega$.
\end{lemma}

\begin{rem}\normalfont
In particular, if we assume that $h$ is harmonic in $M$, we get that the heat kernels are related by \begin{equation}\label{relationofheatkernelsends}\widetilde{p}_{t}(x, y)=\frac{p_{t}(x, y)}{h(x)h(y)}\end{equation} for all $t>0$ and $x, y\in M$.
\end{rem}

\begin{defin}\normalfont Let $\Omega$ be an open set in $M$ and $K$ be a compact set in $\Omega$. Then we call the pair $(K, \Omega)$ a \textit{capacitor} and define the capacity $\textnormal{cap}(K, \Omega)$ by \begin{equation}\label{Defcap}\textnormal{cap}(K, \Omega)=\inf_{\phi\in \mathcal{T}(K, \Omega)}\int_{\Omega}{|\nabla \phi|^{2}d\mu},\end{equation} where $\mathcal{T}(K, \Omega)$ is the set of test functions defined by \begin{equation}\label{testfunkcap}\mathcal{T}(K, \Omega)=\{\phi\in C_{0}^{\infty}(\Omega): \phi|_{K}=1\}.\end{equation}
\end{defin}

Let $\Omega$ be precompact. Then it is known that the Dirichlet integral in (\ref{Defcap}) is minimized by a harmonic function $\varphi$, so that the infimum is attained by the weak solution to the Dirichlet problem in $\Omega\setminus \overline{K}$: $$\left\{ \begin{array}{l} \Delta \varphi=0 \\ \varphi|_{\partial K}=1
\\\varphi|_{\partial \Omega}=0.\\\frac{\partial \varphi}{\partial \nu}|_{\delta (\Omega\setminus \overline{K})}=0 \end{array}\right. $$
The function $\varphi$ is called the \textit{equilibrium potential} of the capacitor $(K, \Omega)$.

We always have the following identity: \begin{equation}\label{fluxvscap}\textnormal{cap}(K, \Omega)=\int_{\Omega}{|\nabla \varphi|^{2}d\mu}=\int_{\Omega\setminus \overline{K}}{|\nabla \varphi|^{2}d\mu}=-\textnormal{flux}(\varphi),
\end{equation} where $\textnormal{flux}(\varphi)$ is defined by $$\textnormal{flux}(\varphi):=\int_{\partial W}{\frac{\partial \varphi}{\partial \nu}d\mu'},$$ where $W$ is any open region in the domain of $\varphi$ with smooth precompact boundary such that $\overline{K}\subset W$ and $\nu$ is the outward normal unit vector field on $\partial W$. By the Green formula (\ref{intbypartsdelta}) and the harmonicity of $\varphi$, $\textnormal{flux}(\varphi)$ does not depend on the choice of $W$.

\begin{defin}\normalfont
We say that a precompact open set $U\subset M$ has locally positive capacity, if there exists a precompact open set $\Omega$ such that $\overline{U}\subset \Omega$ and $\textnormal{cap}(U, \Omega)>0$.
\end{defin}

It is a consequence of the local Poincaré inequality, that if $\textnormal{cap}(U, \Omega)>0$ for some precompact open $\Omega$, then this is true for all precompact open $\Omega$ containing $\overline{U}$.

\begin{lemma}\label{Lemmaconstructionh}
Let $(M, \mu)$ be a complete, non-compact weighted manifold and $K$ be a compact set in $M$ with locally positive capacity and smooth boundary $\partial K$. Fix some $x_{0}\in M$ and set $B_{r}:=B(x_{0}, r)$ for all $r>0$ and assume that $K$ is contained in a ball $B_{r_{0}}$ for some $r_{0}>0$. Let us also set $\Omega=M\setminus K$, so that $\left(\overline{\Omega}, \mu\right)$ becomes a weighted manifold with boundary.
Then there exists a positive smooth function $h$ in $\overline{\Omega}$ that is harmonic in $\Omega$ and satisfies for all $r\geq r_{0}$, \begin{equation}\label{vergleichminicapah}\min_{\partial B_{r}}h\leq C~\textnormal{cap}(K, B_{r})^{-1},\end{equation} for some constant $C>0$. Moreover, the weighted manifold $\left(\overline{\Omega}, \widetilde{\mu}\right)$ is non-parabolic, where measure $\widetilde{\mu}$ on $\overline{\Omega}$ is defined by (\ref{defvonmeasuremith}).\end{lemma}

\begin{proof}
For any $R>r_{0}$, let $\varphi_{R}$ be the equilibrium potential of the capacitor $(K, B_{R})$. It follows from (\ref{fluxvscap}), that \begin{equation}\label{capgleichminsfluxK}\textnormal{cap}(K, B_{R})=-\textnormal{flux}(\varphi_{R}).\end{equation} By our assumption on $K$, we have for all $R> r_{0}$, \begin{equation*}\label{assumptionboundary}\textnormal{cap}(K, B_{R})>0,\end{equation*} whence we can consider the sequence $$v_{R}=\frac{1-\varphi_{R}}{\textnormal{cap}(K, B_{R})}.$$ By (\ref{capgleichminsfluxK}) this sequence satisfies \begin{equation}\label{fluxvReins}\textnormal{flux}(v_{R})=1.\end{equation} Let us extend all $v_{R}$ to $K$ by setting $v_{R}\equiv 0$ on $K$. We claim that for all $R>r>r_{0}$, \begin{equation}\label{vergleichminicapav}\min_{\partial B_{r}}v_{R}\leq \textnormal{cap}(K, B_{r})^{-1}.\end{equation} For $R>r>r_{0}$, denote $m_{r}=\min_{\partial B_{r}}v_{R}$. It follows from the minimum principle and the fact that $v_{R}\equiv 0$ on $K$, that the set $$U_{r}:=\{x\in M:v_{R}(x)<m_{r}\}$$ is inside $B_{r}$ and contains $K$. Then observe that the function $1-\frac{v_{R}}{m_{r}}$ is the equilibrium potential for the capacitor $(K, U_{r})$, whence $$\textnormal{cap}(K, B_{r})\leq \textnormal{cap}(K, U_{r})=\textnormal{flux}\left(\frac{v_{R}}{m_{r}}\right)=\frac{1}{m_{r}},$$ which proves (\ref{vergleichminicapav}).
Note that, since $v_{R}$ vanishes on $\partial \Omega$, the maximum principle implies that for all $R>r>r_{0}$, \begin{equation}\label{infinoderrand}\sup_{B_{r}\setminus K}v_{R}=\max_{\partial B_{r}}v_{R}.\end{equation} Hence, we obtain from (\ref{infinoderrand}), the local elliptic Harnack inequality and (\ref{vergleichminicapav}), that for every $R>r>r_{0}$, \begin{equation}\label{uniformboundvR}\sup_{B_{r}\setminus K}v_{R}\leq C(r)\min_{\partial B_{r}}v_{R}\leq C(r)\textnormal{cap}(K, B_{r})^{-1},\end{equation} where the constant $C(r)$ depends only on $r$. Thus, the bound in (\ref{uniformboundvR}) is uniform in $R$, when $R\gg r$, so that in this case, the sequence $v_{R}$ is uniformly bounded in $B_{r}\setminus K$. Now define $V_{k}:=B_{r_{k}}\setminus K$, where $B_{r_{k}}$ are open balls of radius $r_{k}\geq r_{0}$ with $\lim_{k\to \infty}r_{k}=+\infty$ so that $\{V_{k}\}_{k}$ gives a sequence of precompact open sets that covers $\Omega$.
By using a diagonal process, we obtain a subsequence $v_{R_{k}}$ of $v_{R}$ that converges in all $V_{k}$, and hence, in $\Omega$. In addition, the limit $v:=\lim_{k\to \infty}v_{R_{k}}$ is a harmonic function in $\Omega$.
Since \begin{equation}\label{vRundvgleichinBR}v=v_{R} \quad\textnormal{in}~B_{R}\setminus K,\end{equation} we have $$\int_{\partial B_{r}}{\frac{\partial v}{\partial \nu}d\mu'}=\int_{\partial B_{r}}{\frac{\partial v_{R}}{\partial \nu}d\mu'},$$ which together with (\ref{fluxvReins}) implies \begin{equation}\label{fluxvonhimbeweis} \textnormal{flux}(v)=1,\end{equation} whence $v$ is non-constant. Furthermore, $v$ is non-zero and by (\ref{vRundvgleichinBR}), non-negative in $\overline{\Omega}$. Let us define the function $h=1+v$ in $\overline{\Omega}$ so that $h$ is positive and smooth in $\overline{\Omega}$. Also, it follows from (\ref{vergleichminicapav}), that for all $r>r_{0}$, $$\min_{\partial B_{r}}h\leq 1+\textnormal{cap}(K, B_{r})^{-1}\leq (1+\textnormal{cap}(K, B_{r_{0}}))\textnormal{cap}(K, B_{r})^{-1},$$ which proves (\ref{vergleichminicapah}) with $C=1+\textnormal{cap}(K, B_{r_{0}})$. 
        
Let us now show that the weighted manifold $(\overline{\Omega}, \widetilde{\mu})$ is non-parabolic. For that purpose, consider in $\overline{\Omega}$ the positive smooth function $w=\frac{1}{h}$. Then we have by Lemma \ref{neuesmasdelta}, that function $w$ satisfies in $\Omega$, $$\Delta_{\widetilde{\mu}}(w)=\Delta_{\widetilde{\mu}}\left(\frac{1}{h}\right)=\frac{1}{h}\Delta_{\mu}1=0.$$ so that the function $w$ is $\Delta_{\widetilde{\mu}}$-harmonic in $\Omega$. Observe that \begin{equation}\label{ableitungvonw}\frac{\partial w}{\partial \nu}=-\frac{\partial h}{\partial \nu}{\frac{1}{h^{2}}},\end{equation} where $\nu$ denotes the outward normal unit vector field on $\partial \Omega$. Since $v$ is non-negative in $\Omega$ and $v=0$ on $\partial \Omega$, we have $\frac{\partial h}{\partial \nu}\leq0$ on $\partial \Omega$, whence we get by (\ref{ableitungvonw}), $$\frac{\partial w}{\partial \nu}\geq0\quad \textnormal{on}~ \partial \Omega.$$ Hence, we conclude that $w$ is $\Delta_{\widetilde{\mu}}$-superharmonic in $\overline{\Omega}$, positive and non-constant, which implies that $(\overline{\Omega}, \widetilde{\mu})$ is non-parabolic.
\end{proof}

\begin{rem}\normalfont
Note that the function $h$ constructed in Lemma \ref{Lemmaconstructionh} is $\Delta_{\mu}$-subharmonic in $\overline{\Omega}$. If we assume that the weighted manifold $(\overline{\Omega}, \mu)$ is parabolic, we obtain that $h$ is unbounded since a non-constant bounded subharmonic function can only exist on non-parabolic manifolds.
\end{rem}

\subsection{Locally Harnack case}\label{positiveresultslower}

\begin{defin}\label{deflocallyharnack}\normalfont
The weighted manifold $(M, \mu)$ is said to be a \textit{locally Harnack manifold} if there is $\rho>0$, called the \textit{Harnack radius}, such that for any point $x\in M$ the following is true:
\begin{enumerate}
                \item for any positive numbers $r<R<\rho$ \begin{equation}\label{locHarvD} \frac{V(x, R)}{V(x, r)}\leq a \left(\frac{R}{r}\right)^{n}.\end{equation}
                \item Poincaré inequality: for any Lipschitz function $f$ in the ball $B(x, R)$ of a radius $R<\rho$ we have \begin{equation}\label{zwopoincare}\int_{B(x, R)}{|\nabla u|^{2}d\mu}\geq\frac{b}{R^{2}}\int_{B(x, R/2)}{(f-\overline{f})^{2}d\mu},\end{equation} where we denote $$\overline{f}:=\dashint_{B(x, R/2)}{f}d\mu:=\frac{1}{V(x, R/2)}\int_{B(x, R/2)}{f}d\mu$$
\end{enumerate}
and $a, b$ and $n$ are positive constants and $V(x, r)$ denotes the volume function of $(M, \mu)$.
\end{defin}

For example, the conditions 1. and 2. are true in the case when the manifold $M$ has Ricci curvature bounded below by a (negative) constant $-K$ (see \cite{Buser1982}). Hence, for example, any manifold $M$ of \textit{bounded geometry} is a locally Harnack manifold.

\begin{lemma}[\cite{grigor1994heat}, Theorem 2.1]\label{thmlocallyharnack}
Let $(M, \mu)$ be a locally Harnack manifold. Then we have, for any precompact open set $U\subset M$, \begin{equation}\label{lowerbdlambdaone}\lambda_{1}(U)\geq \frac{c}{\rho^{2}}\min\left(\left(\frac{V_{0}}{\mu(U)}\right)^{2}, \left(\frac{V_{0}}{\mu(U)}\right)^{2/n}\right),\end{equation} where $$V_{0}=\inf_{x\in M}\{V(x, \rho):B(x, \rho)\cap U\ne \emptyset\}$$ and the constant $c$ depends on $a, b, n$ from (\ref{locHarvD}) and (\ref{zwopoincare}).
\end{lemma}

\begin{defin}\normalfont
We say that a manifold $M$ satisfies the \textit{spherical Harnack inequality} if there exist $x_{0}\in M$ and constants $r_{0}>0$, $C_{H}>0$, $N_{H}>0$ and $A>1$, so that for any positive harmonic function $u$ in $M\setminus \overline{B(x_{0}, A^{-1}r)}$ with $r\geq r_{0}$, \begin{equation}\label{annuliharnack2}\sup_{\partial B(x_{0}, r)}u\leq C_{H}r^{N_{H}} \inf_{\partial B(x_{0}, r)}u.\end{equation}
\end{defin}

\textbf{Assumption:} In this section, when considering an end $\Omega$ of a complete non-compact weighted manifold $(M, \mu)$, we always assume that there exists a complete weighted manifold $(M_{0}, \mu_{0})$ and a compact set $K_{0}\subset M_{0}$ that is the closure of a non-empty open set, such that $\Omega$ is isometric to $M_{0}\setminus K_{0}$ in the sense of weighted manifolds. For simplicity and since we only use the intrinsic geometry of $M_{0}$, we denote by $B(x, r)$ the geodesic balls in $M_{0}$ and by $V(x, r)$ the volume function of $M_{0}$. 

\begin{thm}\label{thmlocharsph}
Let $\Omega$ be an end of a complete non-compact weighted manifold $(M, \mu)$. Assume that 
$M_{0}$ is a locally Harnack manifold with Harnack radius $\rho>0$, where $M_{0}$ is defined as above, and that there exists $x_{0}\in M_{0}$ so that
\begin{itemize}
                \item $M_{0}$ satisfies the spherical Harnack inequality (\ref{annuliharnack2}).
                \item $M_{0}$ satisfies the polynomial volume growth condition (\ref{PolygrowthOM}).
                \item There are constants $v_{0}>0$ and $\theta\geq 0$ so that for any $x\in M_{0}$, if $d(x, x_{0})\leq R$ for some $R>\rho$, it holds that \begin{equation}\label{upperboundVOm}V(x, \rho)\geq v_{0}R^{-\theta}.\end{equation}
\end{itemize}
Then, for any $x\in M$, there exist $\alpha>0$, $t_{x}>0$ and $c_{x}>0$ such that for all $t\geq t_{x}$, \begin{equation}\label{lowerbndendparamitharnaanuohneomloc}p_{t}(x, x)\geq \frac{c_{x}}{t^{\alpha}},\end{equation} where $\alpha=\alpha(N, \theta, n, N_{H})$ and $n$ is as in (\ref{locHarvD}).
\end{thm}


\begin{proof}
Let us set $B_{r}=B(x_{0}, r)$ and $V(r)=V(x_{0}, r)$ and $K_{0}$ be contained in a ball $B_{\delta}$ for some $\delta>0$. It follows from ([\cite{heinonen2018nonlinear}, Theorem 2.25]) that $K_{0}$ has locally positive capacity. Then by Lemma \ref{Lemmaconstructionh} there exists a positive smooth function $h$ in $\overline{\Omega}$ that is harmonic in $\Omega$ and such that the weighted manifold $\left(\overline{\Omega}, \widetilde{\mu}\right)$ is non-parabolic, where measure $\widetilde{\mu}$ is defined by (\ref{defvonmeasuremith}). Now, our aim is to apply the estimate (\ref{lowerbndend}) in Theorem \ref{thmlowerptbd} to the weighted manifold $(\overline{\Omega}, \widetilde{\mu})$. For that purpose, it is sufficient to show that there are positive constants $\widetilde{r_{0}}, \widetilde{C}$ and $\widetilde{N}>2$ such that for all $r\geq \widetilde{r_{0}}$, \begin{equation}\label{Volumegrowththeotilde2}\widetilde{V}_{\Omega}(r)=\int_{B_{r}\cap \Omega}{h^{2}d\mu}\leq \widetilde{C}r^{\widetilde{N}}.\end{equation}
Firstly, by (\ref{vergleichminicapah}), there is a constant $C_{\delta}>0$ such that for all $r\geq \delta$, 
\begin{equation}\label{kettemitcapunmax2l}\min_{\partial B_{r}}h\leq C_{\delta}\textnormal{cap}(K_{0}, B_{r})^{-1}.\end{equation}
As $h$ is harmonic in $M_{0}\setminus \overline{B_{\delta}}$, the hypothesis (\ref{annuliharnack2}) implies that there exists a constant $C_{H}>0$, so that for every $r\geq \max(r_{0}, A\delta)$, $$\max_{\partial B_{r}}h\leq C_{H}r^{N_{H}}\min_{\partial B_{r}}h.$$
Combining this with (\ref{kettemitcapunmax2l}), we obtain for all $r\geq \max(r_{0}, A\delta)$ with $C_{0}=C_{H}C_{\delta}$, \begin{equation}\label{Folgerungfuerh2l}\max_{\partial B_{r}}h\leq C_{0}r^{N_{H}}\textnormal{cap}(K_{0}, B_{r})^{-1}.\end{equation}
For any $r\geq\delta$, let $\varphi_{r}$ be the equilibrium potential of the capacitor $(K_{0}, B_{r})$. Since $$\int_{B_{r}}{|\nabla \varphi_{r}|^{2}d\mu_{0}}=\textnormal{cap}(K_{0}, B_{r})$$ and $$\int_{B_{r}}{\varphi_{r}^{2}d\mu_{0}}\geq \mu_{0}(K_{0}),$$ we obtain $$\lambda_{1}(B_{r})\leq \frac{\int_{B_{r}}{|\nabla \varphi_{r}|^{2}d\mu_{0}}}{\int_{B_{r}}{\varphi_{r}^{2}d\mu_{0}}}\leq\frac{\textnormal{cap}(K_{0}, B_{r})}{\mu(K_{0})},$$ whence, together with (\ref{Folgerungfuerh2l}), we deduce  \begin{equation}\label{vergleichcapla1}\max_{\partial B_{r}}h\leq C_{0}\mu(K_{0})r^{N_{H}}\lambda_{1}(B_{r})^{-1}.\end{equation} 
Since $M_{0}$ is a locally Harnack manifold, we can apply Lemma \ref{thmlocallyharnack} and obtain from (\ref{lowerbdlambdaone}), that for all $r\geq \delta$, \begin{equation}\label{anwendunglowebdlambdae}\lambda_{1}(B_{r})\geq \frac{c}{\rho^{2}}\min\left(\left(\frac{V_{0}}{V(r)}\right)^{2}, \left(\frac{V_{0}}{V(r)}\right)^{2/n}\right),\end{equation} where $$V_{0}=\inf_{x\in M_{0}}\{V(x, \rho):B(x, \rho)\cap B_{r}\ne \emptyset\}.$$
Note that the condition $B(x, \rho)\cap B_{r}\ne \emptyset$ implies that $d(x_{0}, x)\leq r+\rho$. Thus, we obtain from the hypothesis (\ref{upperboundVOm}), assuming $r\geq\rho$, $$V(x, \rho)\geq v_{0}(r+\rho)^{-\theta}\geq v_{0}2^{-\theta}r^{-\theta}.$$
Therefore, we have for all $r\geq \rho$, $$V_{0}\geq C_{\theta}r^{-\theta},$$ with $C_{\theta}=v_{0}2^{-\theta}$. Hence, using the polynomial volume growth condition (\ref{PolygrowthOM}), we obtain from (\ref{anwendunglowebdlambdae}), that for all $r\geq\max(r_{0}, \rho, A\delta)$, $$\lambda_{1}(B_{r})\geq C_{1}\min \left(r^{-2(N+\theta)}, r^{-2(N+\theta)/n}\right),$$ where $$C_{1}=\frac{c}{\rho^{2}}\min\left(\left(\frac{C_{\theta}}{C}\right)^{2}, \left(\frac{C_{\theta}}{C}\right)^{2/n}\right),$$ so that by setting \begin{equation}\label{defvonbetfornh}\beta=2\max\left(N+\theta, \frac{N+\theta}{n}\right),\end{equation} we deduce for $r\geq \max(r_{0}, \rho, A\delta, 1)$, $$\lambda_{1}(B_{r})\geq C_{1}r^{-\beta}.$$
Combining this with (\ref{vergleichcapla1}), we obtain for every $r\geq \max(r_{0}, \rho, A\delta, 1)$, \begin{equation}\label{fuervolgrowth2l}\max_{\partial B_{r}}h\leq C_{2}r^{\beta+N_{H}},\end{equation} where $$C_{2}=C_{0}C_{1}^{-1}\mu_{0}(K_{0})^{-1}.$$
Hence, (\ref{fuervolgrowth2l}), the polynomial volume growth condition (\ref{PolygrowthOM}) and the maximum principle imply that for all $r\geq \max(r_{0}, \rho, A\delta, 1)$, $$\widetilde{V}_{\Omega}(r)=\int_{B_{r}\cap \Omega}{h^{2}d\mu}\leq V(r)\max_{\partial B_{r}}h^{2}\leq C_{2}^{2}Cr^{N+2(\beta+N_{H})},$$ which proves (\ref{Volumegrowththeotilde2}) with $\widetilde{r}_{0}=\max(r_{0}, \rho, A\delta, 1)$, $\widetilde{N}=2(\beta+N_{H})+N$ and $\widetilde{C}=C_{2}^{2}C,$ and implies that the weighted manifold $(\Omega, \widetilde{\mu})$ has a polynomial volume growth.
Thus, the hypotheses of Theorem \ref{thmlowerptbd} are fulfilled and we obtain by (\ref{lowerbndend}), that for any $x\in \Omega$, there exist $\widetilde{t}_{x}>0$ and $\widetilde{c}_{x}>0$, such that for all $t\geq \widetilde{t}_{x}$, $$\widetilde{p}_{t}^{\Omega}(x, x)\geq \frac{\widetilde{c}_{x}}{(t\log t)^{\beta+N_{H}+N/2}},$$ where $\beta$ is defined by (\ref{defvonbetfornh}). Since $h$ is harmonic in $\Omega$, we therefore conclude by (\ref{relhtransformheat}) that for any $x\in \Omega$ and all $t\geq \widetilde{t}_{x}$, $$p^{\Omega}_{t}(x, x)=h^{2}(x)\widetilde{p}^{\Omega}_{t}(x, x)\geq \frac{\widetilde{c}_{x}h^{2}(x)}{(t\log t)^{\beta+N_{H}+N/2}},$$ which yields (\ref{lowerbndendparamitharnaanuohneomloc}) for all $x\in M$ by using $p_{t}^{\Omega}\leq p_{t}$ and by means of the local parabolic Harnack inequality.
\end{proof}

\begin{rem}\normalfont
Note that it follows from the non-parabolicity of $\left(\overline{\Omega}, \widetilde{\mu}\right)$, that $4\max\left(N+\theta, \frac{N+\theta}{n}\right)+2N_{H}+N>2$.
\end{rem}

\subsection{End with relatively connected annuli}

\begin{defin}\normalfont We say that a manifold $M$ with fixed point $x_{0}\in M$ satisfies \textit{the relatively connected annuli condition (RCA)} if there exists $A>1$ such that, for any $r>A^{2}$ and all $x, y$ with $d(x_{0}, x)=d(x_{0}, y)=r$, there exists a continuous path $\gamma:[0, 1]\to M$ with $\gamma(0)=x$ and $\gamma(1)=y$, whose image is contained in $B(x_{0}, Ar)\setminus B(x_{0}, A^{-1}r)$.
\end{defin} 

\begin{rem}\normalfont
Note that, even though the condition (RCA) is formulated for the specific point $x_{0}$, it is equivalent to the (RCA) condition with respect to any other point $x_{1}$ with possibly a different constant $A$.  
\end{rem}

\begin{example}\normalfont
Any \textit{Riemannian model} (see Subsection \ref{exindimtwo} and Section \ref{secisobdry}) with dimension $n\geq 2$ has relatively connected annuli.
\end{example}

\begin{cor}\label{thmlowerptbdgeo}
Let $\Omega$ be an end of a complete non-compact weighted manifold $(M, \mu)$ and assume that 
$M_{0}$ is a locally Harnack manifold with Harnack radius $\rho>0$, where $M_{0}$ is defined as above. Also assume that there exists $x_{0}\in M_{0}$ so that
\begin{itemize}
                \item $M_{0}$ satisfies (RCA) with some constant $A>1$.
        \item There exist constants $L>0$ and $C>0$ so that for all $r\geq L$, \begin{equation}\label{volumeannuli}V(Ar)-V(A^{-1}r)\leq C\log r,\end{equation} where we denote $V(r)=V(x_{0}, r)$.
        \item There exists a constant $v_{0}>0$ such that for any $y\in M_{0}$, \begin{equation}\label{lowerboundball}V(y, \rho/3)\geq v_{0}.\end{equation}
\end{itemize}Then, for any $x\in M$, there exist $\alpha>0$, $t_{x}>0$ and $c_{x}>0$ such that for all $t\geq t_{x}$, \begin{equation}\label{lowerbndendpara}p_{t}(x, x)\geq \frac{c_{x}}{t^{\alpha}},\end{equation} where $\alpha=\alpha(n, v_{0}, \rho, C)$.
\end{cor}

\begin{proof}
As before, we denote $B_{r}=B(x_{0}, r)$.
Obviously, the hypothesis (\ref{lowerboundball}) implies the condition (\ref{upperboundVOm}) with $\theta=0$.
Hence, to apply Theorem \ref{thmlocharsph}, it remains to show that $M_{0}$ has a polynomial volume growth as in (\ref{upperboundVOm}) and $M_{0}$ satisfies the spherical Harnack inequality (\ref{annuliharnack2}).
The polynomial volume growth condition (\ref{upperboundVOm}) follows from (\ref{volumeannuli}).
        
Let us now prove that the spherical Harnack inequality (\ref{annuliharnack2}) holds in $M_{0}$. Assume that $r\geq L$ and cover the set $B_{Ar}\setminus B_{ A^{-1}r}$, with balls $B(x_{i}, \rho/3)$ where $x_{i}\in M_{0}$ and $A>1$ is as in (RCA). By applying the Banach process, there exists a number $\tau(r)$ and a subsequence of disjoint balls $\{B(x_{i_{k}}, \rho/3)\}_{k=1}^{\tau(r)}$ such that the union of the balls $\{B(x_{i_{k}}, \rho)\}_{k=1}^{\tau(r)}$ cover the set $B_{Ar}\setminus B_{A^{-1}r}$. Hence, it follows from (\ref{volumeannuli}), that \begin{equation}\label{summederballsabsch}\sum_{i=1}^{\tau(r)}{V(x_{i}, \rho/3)}\leq V( Ar)-V(A^{-1}r)\leq C\log r.\end{equation} Then the hypothesis (\ref{lowerboundball}), combined with (\ref{summederballsabsch}), implies that \begin{equation}\label{anzahlderballschain}\tau(r)\leq \frac{C\log r}{v_{0}}.\end{equation} Let $y_{1}, y_{2}$ be two points on $\partial B_{r}$ such that $\min_{\partial B_{r}}u=u(y_{1})$ and $\max_{\partial B_{r}}u=u(y_{2})$ and $\gamma$ be a continuous path connecting them in $B_{Ar}\setminus B_{A^{-1}r}$ as is it ensured by (RCA) for all $r>A^{2}$. Now select out of the sequence $\{B(x_{i_{k}}, \rho)\}_{k=1}^{\tau(r)}$ those balls that intersect $\gamma$. In this way, we obtain a chain of at most $\tau(r)$ balls, which connect $y_{1}$ and $y_{2}$. Now let $u$ be a positive harmonic function in $M_{0}\setminus\overline{B_{A_{0}^{-1}r}}$, where $A_{0}\geq A$ is such that any ball of this chain lies in $M_{0}\setminus \overline{B_{A_{0}^{-1}r}}$ for all $1\leq i\leq \tau(r)$ and $r>A_{0}^{2}$. Applying the local elliptic Harnack inequality to $u$ repeatedly in the balls of this chain, we obtain $$\max_{\partial B_{r}}u=u(y_{2})\leq (C_{\rho})^{\tau}u(y_{1})=(C_{\rho})^{\tau}\min_{\partial B_{r}}u,$$ where $C_{\rho}$ is the Harnack constant in all $B(x_{i_{k}}, \rho)$. Together with (\ref{anzahlderballschain}), this yields $$\max_{\partial B_{r}}u\leq r^{\frac{c}{v_{0}}\log C_{\rho}}\min_{\partial \partial B_{r}}u,$$ which proves the spherical Harnack inequality (\ref{annuliharnack2}) with $N_{H}=\frac{C}{v_{0}}\log C_{\rho}$. Thus the hypotheses of Theorem \ref{thmlocharsph} are fulfilled and we obtain from (\ref{lowerbndendparamitharnaanuohneomloc}), that for any $x\in M$, there exist $t_{x}>0$, $c_{x}>0$ and $\alpha>0$ such that for all $t\geq t_{x}$, $$p_{t}(x, x)\geq \frac{c_{x}}{t^{\alpha}},$$ where $\alpha=\alpha(n, N_{H})$, which finishes the proof.
\end{proof}


\begin{defin}\normalfont
As usual, for any piecewise $C^{1}$ path $\gamma:I\to M$, where $I$ is an interval in $\mathbb{R}$, denote by $l(\gamma)$ the length of $\gamma$ defined by $$l(\gamma)=\int_{I}{|\dot{\gamma}(t)|dt},$$ where $\dot{\gamma}$ is the velocity of $\gamma$, given by $\dot{\gamma}(t)(f)=\frac{d}{dt}f(\gamma(t))$ for any $f\in C^{\infty}(M)$.   
\end{defin}

\begin{cor}\label{thmlowerptbdgeo2}
Let $\Omega$ be an end of a complete non-compact weighted manifold $(M, \mu)$ and assume that 
for some $\kappa\geq0$, we have \begin{equation}\label{ricciincorlow}Ric(M_{0})\geq -\kappa,\end{equation} where $M_{0}$ is defined as above. Suppose that there exists $x_{0}\in M_{0}$ so that 
\begin{itemize}
                \item $M_{0}$ satisfies (RCA) with $A>1$ and piecewise $C^{1}$ path $\gamma$ so that there is some constant $c>0$ such that for all $r>A^{2}$, \begin{equation}\label{laengegeodesicbound}l(\gamma)\leq c\log r.\end{equation}
                \item There are constants $v_{0}>0$ and $\theta\geq 0$ so that for any $y\in M_{0}$, if $d(y, x_{0})\leq R$ for some $R>1$, it holds that $$V(y, \rho)\geq v_{0}R^{-\theta}.$$
\end{itemize}Then, for any $x\in M$, there exist $\alpha>0$, $t_{x}>0$ and $c_{x}>0$ such that for all $t\geq t_{x}$, \begin{equation}\label{lowerbndendparavol}p_{t}(x, x)\geq \frac{c_{x}}{t^{\alpha}},\end{equation} where $\alpha=\alpha(c, \theta, \kappa)$.
\end{cor}
\begin{proof}
The assumption (\ref{ricciincorlow}) implies that $M_{0}$ is a locally Harnack manifold. Hence we are left to show that $M_{0}$ has a polynomial volume growth as in (\ref{PolygrowthOM}) and satisfies the spherical Harnack inequality (\ref{annuliharnack2}) to apply Theorem \ref{thmlocharsph}.
Again we denote $B_{r}=B(x_{0}, r)$ and $V(r)=V(x_{0}, r)$.
By the Bishop-Gromov theorem, the hypothesis (\ref{ricciincorlow}) implies that there exists a constant $C_{\kappa}>$, so that for any $y\in M_{0}$ and $R>1$, \begin{equation}\label{Volumeboundexp}V(y, R)\leq e^{C_{\kappa}R}.\end{equation}
Together with the assumption (\ref{laengegeodesicbound}), this yields that the polynomial volume growth condition (\ref{PolygrowthOM}) holds in $M_{0}$.
        
Let us now show that $M_{0}$ satisfies the spherical Harnack inequality (\ref{annuliharnack2}). Let $A>1$ be as above and assume that $r>A^{2}$. Denote by $y_{1}, y_{2}$ the points on $\partial B_{r}$ such that $\min_{\partial B_{r}}u=u(y_{1})$ and $\max_{\partial B_{r}}u=u(y_{2})$ and let $\gamma$ be a continuous path connecting them in $B_{Ar}\setminus B_{A^{-1}r}$ as is it ensured by (RCA). Then cover the path $\gamma$ with balls $\{B(x_{i}, \rho)\}_{i=1}^{\tau(r)}$, where $x_{i}\in M_{0}$ and $\rho>0$. Now let $u$ be a positive harmonic function in $M_{0}\setminus\overline{B_{A_{0}^{-1}r}}$, where $A_{0}\geq A$ is such that $B(x_{i}, \rho)\subset M_{0}\setminus \overline{B_{A_{0}^{-1}r}}$ for all $1\leq i\leq \tau(r)$ and $r>A_{0}^{2}$. In this way, we obtain a chain of at most $\tau(r)$ balls $B(x_{i}, \rho)$, which connect $y_{1}$ and $y_{2}$. By (\ref{laengegeodesicbound}), we deduce that \begin{equation}\label{anzahlballsinsquence} \tau(r)\leq \frac{c}{\rho}\log(r).\end{equation} Applyig the local elliptic Harnack inequality to $u$ repeatedly in the balls of this chain, we obtain $$\max_{\partial B_{r}}u=u(y_{2})\leq (C_{\rho})^{\tau}u(y_{1})=(C_{\rho})^{\tau}\min_{\partial B_{r}}u,$$ where $C_{\rho}$ is the Harnack constant in all $B(x_{i}, \rho)$. Together with (\ref{anzahlballsinsquence}), this yields $$\max_{\partial B_{r}}u\leq r^{\frac{c}{\rho}\log C_{\rho}}\min_{\partial B_{r}}u,$$ which proves (\ref{annuliharnack2}) with $N_{H}=\frac{c}{\rho}\log C_{\rho}$.
Thus the hypotheses of Theorem \ref{thmlocharsph} are fulfilled and we obtain by (\ref{lowerbndendparamitharnaanuohneomloc}), that for any $x\in M$, there exist $t_{x}>0$, $c_{x}>0$ and $\alpha>0$ such that for all $t\geq t_{x}$, $$p_{t}(x, x)\geq \frac{c_{x}}{t^{\alpha}},$$ which finishes the proof.
\end{proof}

\subsection{An example in dimension two}\label{exindimtwo}

Consider the topological space $M=(0, +\infty)\times \mathbb{S}^{1}$, that is, any point $x\in M$ can be represented in the polar coordinates $x=(r, \theta)$ with $r>0$ and $\theta\in \mathbb{S}^{1}$. Equip $M$ with the Riemannian metric $ds^{2}$ given by \begin{equation}\label{metricmodelman}ds^{2}=dr^{2}+\psi^{2}(r)d\theta^{2},\end{equation} where $\psi(r)$ is a smooth positive function on $(0, +\infty)$ and $d\theta^{2}$ is the standard Riemannian metric on $\mathbb{S}^{1}$. In this case, $M$ is called a \textit{two-dimensional Riemannian model with a pole}.

\begin{rem}\normalfont
A sufficient and necessary condition, for the existence of this manifold is that $\psi$ satisfies the conditions $\psi(0)=0$ and $\psi'(0)=1$. This ensures that the metric $ds^{2}$ can be smoothly extended to the origin $r=0$ (see \cite{greene2006function}).
\end{rem}

We define the area function $S$ on $(0, +\infty)$ by \begin{equation}\label{defofarea}S(r)=\psi(r).\end{equation}

\begin{prop}\label{propforspherintwo}
Let $M$ be a two-dimensional Riemannian model with a pole. Suppose that for any $A>1$, there exists a constant $c>0$, so that for all large enough $r$, \begin{equation}\label{langsamevariquo}\sup_{t\in (A^{-1}r, Ar)}\frac{S''_{+}(t)}{S(t)}\leq c\frac{S''_{+}(r)}{S(r)}.\end{equation} Also assume that there exists a constant $N>0$ such that, for every large enough $r$,
\begin{equation}\label{spherhartwodimmod}\frac{S(r)}{r}+\sqrt{S''_{+}(r)S(r)}\leq N\log(r).\end{equation}
Then the spherical Harnack inequality (\ref{annuliharnack2}) holds in $M$.
\end{prop}


\begin{proof}
Fix some $x_{0}\in M$ and denote $B_{r}=B(x_{0}, r)$. Since any model manifold of dimension $n\geq 2$ satisfies the (RCA) condition, there exists $A_{0}>1$ such that for all $r>A_{0}^{2}$ and any $x_{1}, x_{2}\in \partial B_{r}$, there exists $T>0$ and a continuous path $\gamma:[0, T]\to M$ such that $\gamma(0)=x_{1}$ and $\gamma(T)=x_{2}$, whose image is contained in $B_{A_{0}r}\setminus B_{A_{0}^{-1}r}$. Let us choose $A>A_{0}$ so that there exists a constant $\epsilon>0$, such that $B(x, R)\subset B_{Ar}\setminus \overline{B_{A^{-1}r}}$, for any $x\in \gamma([0, T])$, where $R=\epsilon r$. Let $u$ be a positive harmonic function in $M\setminus\overline{B_{A^{-1}r}}$ and $x_{1},x_{2}\in \partial B_{r}$ such that $\max_{\partial B_{r}}u=u(x_{1})$ and $\min_{\partial B_{r}}u=u(x_{2})$. Thus, we have to show that there are constants $N_{H}>0$ and $C_{H}>0$, so that if $r$ is large enough, then \begin{equation}\label{sphericalharnackmodtwo}u(x_{1})\leq C_{H}r^{N_{H}}u(x_{2}).\end{equation}
Let $x\in \gamma([0, T])$. Recall from [\cite{Grigoryan2012}, Exercise 3.31], that the Ricci curvature $Ric$ on $M$ is given by \begin{equation}\label{gaussin2S}Ric=-\frac{S''}{S}.\end{equation}
Hence, we obtain from (\ref{gaussin2S}), $$Ric(x)\geq\inf_{t\in (A^{-1}r, Ar)}\left(-\frac{S''(t)}{S(t)}\right)\geq-\sup_{t\in (A^{-1}r, Ar)}\left(\frac{S''_{+}(t)}{S(t)}\right).$$ By (\ref{langsamevariquo}), we get, assuming that $r$ is large enough, \begin{equation}\label{riccigaussgamma}Ric(x)\geq-c\frac{S''_{+}(r)}{S(r)}=:-\kappa(r).\end{equation} Clearly, we can assume that $|\gamma'(t)|=1$. We have $$\int_{0}^{T}{\frac{|\nabla u(\gamma(t))|}{u(\gamma(t))}dt}\leq \sup_{0\leq t\leq T}\frac{|\nabla u(\gamma(t))|}{u(\gamma(t))}\int_{0}^{T}{dt}\leq \sup_{0\leq t\leq T}\frac{|\nabla u(\gamma(t))|}{u(\gamma(t))}d(x_{1}, x_{2}).$$ Again, since $M$ has dimension $n=2$, and as $x_{1},x_{2}\in \partial B_{r}$, we see that 
$$d(x_{1}, x_{2})\leq S(r),$$ whence $$\int_{0}^{T}{\frac{|\nabla u(\gamma(t))|}{u(\gamma(t))}dt}\leq \sup_{0\leq t\leq T}\frac{|\nabla u(\gamma(t))|}{u(\gamma(t))}S(r).$$
Applying the well-known gradient estimate (cf. \cite{cheng1975differential}) to the harmonic function $u$ in all balls $B(x, R)$, we obtain, $$\sup_{0\leq t\leq T}\frac{|\nabla u(\gamma(t))|}{u(\gamma(t))}\leq C_{n}\left(\frac{1+R\sqrt{\kappa(r)}}{R}\right),$$ where $\kappa(r)$ is given by (\ref{riccigaussgamma}) and $C_{n}>0$ is a constant depending only on $n$.
Therefore, we deduce \begin{align*}\log u(x_{1})-\log u(x_{2})=\left|\int_{0}^{T}{\frac{d\log u(\gamma(t))}{dt}} \right|&\leq\int_{0}^{T}{\frac{|du(\gamma(t))|}{u(\gamma(t))}}\\&=\int_{0}^{T}{\frac{|\langle \nabla u, \gamma'(t)\rangle|}{u(\gamma(t))}dt}\\&\leq\int_{0}^{T}{\frac{|\nabla u(\gamma(t))|}{u(\gamma(t))}dt}\\&\leq C_{n}\left(\frac{1}{\epsilon r}+\sqrt{\kappa(r)}\right)S(r),\end{align*} which is equivalent to $$u(x_{1})\leq \exp\left(C_{n}\left(\frac{S(r)}{\epsilon r}+S(r)\sqrt{\kappa(r)}\right)\right)u(x_{2}).$$ Hence, we get by (\ref{riccigaussgamma}),  $$u(x_{1})\leq \exp\left(C_{n}\left(\frac{S(r)}{\epsilon r}+\sqrt{cS''_{+}(r)S(r)}\right)\right)u(x_{2}).$$
Finally, by (\ref{spherhartwodimmod}), we deduce for large enough $r$, $$u(x_{1})\leq r^{C_{n}\max\left\{\sqrt{c}, \frac{1}{\epsilon}\right\}N}u(x_{2}),$$ which proves (\ref{sphericalharnackmodtwo}) with $C_{H}=1$ and $N_{H}=C_{n}\max\left\{\sqrt{c}, \frac{1}{\epsilon}\right\}N$ and finishes the proof.
\end{proof}

\begin{example}\normalfont
Let $(M, \mu)$ be a two-dimensional weighted manifold with end $\Omega$ and, following the notation in Theorem \ref{thmlocharsph}, suppose that $M_{0}$ is a Riemannian model with a pole such that $$S_{0}(r)=\left\{ \begin{array}{lc} r\log r,&r\geq 2 \\ r,&r\leq 1.\end{array}\right.$$  
Let us show that $M_{0}$ satisfies the hypotheses of Theorem \ref{thmlocharsph} so that for any $x\in M$, there exist $t_{x}>0$, $c_{x}>0$ and $\alpha>0$ such that for all $t\geq t_{x}$, \begin{equation}\label{lowermodelheatfulinex}p_{t}(x, x)\geq \frac{c_{x}}{t^{\alpha}}.\end{equation}
Since $S_{0}''(r)=\frac{1}{r}$ for $r\geq 2$, the inequality (\ref{langsamevariquo}) is satisfied and also $$\frac{S_{0}(r)}{r}+\sqrt{(S_{0}'')_{+}(r)S_{0}(r)}=\log r+\sqrt{\log r}\leq 2 \log r,$$ whence (\ref{spherhartwodimmod}) holds and we get that $M_{0}$ satisfies the spherical Harnack inequality (\ref{annuliharnack2}). On the other hand, we have for $r\geq 2$, $-\frac{S_{0}''(r)}{S_{0}(r)}=-\frac{1}{r^{2}\log r}$ so that it follows from (\ref{gaussin2S}) that $M_{0}$ has non-positive bounded below sectional curvature. Hence, $M_{0}$ is a locally Harnack manifold and, as it is simply connected, is a Cartan-Hadamard manifold which yields that the balls in $M_{0}$ of have at least euclidean volume. Therefore, condition (\ref{upperboundVOm}) holds as well and we conclude from Theorem \ref{thmlocharsph} that $(M, \mu)$ admits the estimate (\ref{lowermodelheatfulinex}).
\end{example}

\section{Isoperimetric inequalities for warped products}\label{Isoperimetric inequalities for warped products}

Let $(M, \mu)$ be a weighted manifold. 
\begin{defin}\normalfont
For any Borel set $A\subset M$, define its perimeter $\mu^{+}(A)$ by $$\mu^{+}(A)=\liminf_{r\to 0^{+}}\frac{\mu(A^{r})-\mu(A)}{r},$$ where $A^{r}$ is the $r$-neighborhood of $A$ with respect to the Riemannian metric of $M$.
\end{defin}

\begin{defin}\normalfont We say that $(M, \mu)$ admits the \textit{lower isoperimetric function} $J$ if, for any precompact open set $U\subset M$ with smooth boundary, \begin{equation}\label{isoperimetricdefJ}\mu^{+}(U)\geq J(\mu(U)).\end{equation}
\end{defin}

For example, the euclidean space $\mathbb{R}^{n}$ with the Lebesgue measure satisfies the inequality in (\ref{isoperimetricdefJ}) with the function $J(v)=c_{n}v^{\frac{n-1}{n}}$.

\subsection{Setting and main theorem}\label{settingandmainiso}
Let $(M_{1}, \mu_{1})$ and $(M_{2}, \mu_{2})$ be weighted manifolds and let $M=M_{1}\times M_{2}$ be the direct product of $M_{1}$ and $M_{2}$ as topological spaces. This means that any point $z\in M$ can be written as $z=(x, y)$ with $x\in M_{1}$ and $y\in M_{2}$. Then we define the Riemannian metric $ds^{2}$ on $M$ by \begin{equation}\label{riemmetriconM}ds^{2}=dx^{2}+\psi^{2}(x)dy^{2},\end{equation} where $\psi$ is a smooth positive function on $M_{1}$ and $dx^{2}$ and $dy^{2}$ denote the Riemannian metrics on $M_{1}$ and $M_{2}$, respectively. Let us define the measure $\mu$ on $M$ by \begin{equation}\label{defofmeasonproduct}\mu=\mu_{1}\times \mu_{2}\end{equation} and note that $(M, \mu)$ is then a weighted manifold with respect to the metric in (\ref{riemmetriconM}).

Denote by $\nabla$ the gradient on $M$ and with $\nabla_{x}$ and $\nabla_{y}$ the gradients on $M_{1}$ and $M_{2}$, respectively. It follows from (\ref{riemmetriconM}), that we have the identity \begin{equation}\label{identitygrad}|\nabla u|^{2}=|\nabla_{x}u|^{2}+\frac{1}{\psi^{2}(x)}|\nabla_{y}u|^{2},\end{equation} for any smooth function $u$ on $M$.

\begin{defin}\normalfont
Let $\varphi:(0, +\infty)\to (0, +\infty)$ be a monotone decreasing function. Then we define the generalized inverse function $\phi$ of $\varphi$ on $(0, +\infty)$ by \begin{equation}\label{defgeneinv}\phi(s)=\sup\{t>0:\varphi(t)>s\}.\end{equation}
\end{defin}

\begin{lemma}[\cite{de2015study}, Proposition 4.2] The generalized inverse $\phi$ of $\varphi$ has the following properties:
\begin{enumerate}
                \item $\phi$ is monotone decreasing, right continuous and $\lim_{s\to \infty}\phi(s)=0$;
                \item $\varphi$ is right continuous if and only if $\varphi$ itself is the generalized function of $\phi$, that is \begin{equation}\label{itselfgeneralized}\varphi(t)=\sup\{s>0:\phi(s)>t\};\end{equation}
                \item we have the identity \begin{equation}\label{genehavesameint}\int_{0}^{\infty} {\varphi(t) dt} = \int_{0}^{\infty}{\phi(s)ds}.\end{equation}
\end{enumerate} 
\end{lemma}

\begin{lemma}[\cite{federer2014geometric}]\label{approxofboundary}
Let $U$ be a precompact open subset of a weighted manifold $(M, \mu)$ with smooth boundary. Then $$\mu^{+}(U)=\inf_{\{u_{n}\}}\limsup_{n\to \infty}\int_{M}{|\nabla u_{n}|d\mu}=\inf_{\{u_{n}\}}\liminf_{n\to \infty}\int_{M}{|\nabla u_{n}|d\mu},$$ where $\{u_{n}\}_{n\in \mathbb{M}}$ is a monotone increasing sequence of smooth non-negative functions with compact support, converging pointwise to the characteristic function of the set $U$.
\end{lemma}

The proof of the following theorem follows the ideas of Theorem 1 in \cite{grigor1985isoperimetric}, where an isoperimetric inequality is obtained for Riemannian products $M=M_{1}\times M_{2}$ of two Riemannian manifolds $M_{1}$ and $M_{2}$.

\begin{thm}\label{thm1iso}
Let $(M_{1}, \mu_{1})$ and $(M_{2}, \mu_{2})$ be weighted manifolds and let the weighted manifold $(M, \mu)$ be defined as above, that is, the Riemannian metric on $M$ is defined by (\ref{riemmetriconM}) and measure $\mu$ is defined by (\ref{defofmeasonproduct}). Assume that there exists a constant $C_{0}>0$, such that for all $x\in M_{1}$, \begin{equation}\label{assumponpsi}\psi(x)\leq C_{0}.\end{equation}
Suppose that $(M_{1}, \mu_{1})$ and $(M_{2}, \mu_{2})$ have the lower isoperimetric functions $J_{1}$ and $J_{2}$, which are continuous on the intervals $(0, \mu_{1}(M_{1}))$ and $(0, \mu_{2}(M_{2}))$, respectively. Then $(M, \mu)$ admits the lower isoperimetric function $J$, defined by $$J(v)=c\inf_{\varphi, \phi}\left(\int_{0}^{\infty} {J_{1}(\varphi(t))dt}+\int_{0}^{\infty}{J_{2}(\phi(s))ds}.\right),$$ where $c=\frac{1}{2}\min\left\{1,\frac{1}{C_{0}}\right\}$ and $\varphi$ and $\phi$ are generalized mutually inverse functions such that \begin{equation}\label{varphipsileq}\varphi\leq \mu_{1}(M_{1}), \quad\phi\leq \mu_{2}(M_{2}),\end{equation} and \begin{equation}\label{intvarphipsieq}v=\int_{0}^{\infty} {\varphi(t) dt} = \int_{0}^{\infty}{\phi(s)ds}.\end{equation}
\end{thm}

\begin{proof}
Let $U$ be an open precompact set in $M$ with smooth boundary such that $\mu(U)=v$. Let us define the function \begin{equation}\label{defvonisoonM}I(v)=\inf_{\varphi, \phi}\left(\int_{0}^{\infty} {J_{1}(\varphi(t))dt}+\int_{0}^{\infty}{J_{2}(\phi(s))ds}.\right),\end{equation} where $\varphi$ and $\phi$ are generalized mutually inverse functions satisfying (\ref{varphipsileq}) and (\ref{intvarphipsieq}). We need to prove that \begin{equation}\label{isomitJimproof}\mu^{+}(U)\geq cI(v),\end{equation} where $I$ is defined by (\ref{defvonisoonM}) and $c$ is defined as above. Let $\{f_{n}\}_{n\in \mathbb{N}}$ be a monotone increasing sequence of smooth non-negative functions on $M$ with compact support such that $f_{n}\to 1_{U}$ as $n\to \infty$. Note that by Lemma \ref{approxofboundary}, it suffices to show that  \begin{equation}\label{limitwithperimeter}\limsup_{n\to \infty}\int_{M}{|\nabla f_{n}|d\mu}\geq cI(v).\end{equation} By the identity (\ref{identitygrad}) and using (\ref{assumponpsi}), we have $$|\nabla f_{n}|^{2}=|\nabla_{x}f_{n}|^{2}+\frac{1}{\psi(x)^{2}}|\nabla_{y}f_{n}|^{2}\geq \frac{1}{2}\min\left\{1,\frac{1}{C_{0}}\right\}^{2}\left(|\nabla_{x}f_{n}|+|\nabla_{y}f_{n}|\right)^{2}.$$ Together with (\ref{limitwithperimeter}), it therefore suffices to prove that \begin{equation}\label{whattoshow}\limsup_{n\to \infty}\int_{M}{|\nabla_{x} f_{n}|d\mu}+\limsup_{n\to \infty}\int_{M}{|\nabla_{y} f_{n}|d\mu}\geq I(v).\end{equation} 
Let us first estimate the second summand on the left-hand side of (\ref{whattoshow}). For that purpose, consider for every $x\in M_{1}$, the section $$U_{x}=\{y\in M_{2}: (x, y)\in U\}.$$ By Sard's theorem, the set $U_{x}$ has smooth boundary for allmost all $x$. Considering the function $f_{n}(x, y)$ as a function on $M_{2}$ with fixed $x\in M_{1}$, we obtain by Lemma \ref{approxofboundary} for allmost all $x$, \begin{equation}\label{approxwithsecone} \liminf_{n\to \infty}\int_{M}{|\nabla_{y} f_{n}(x, y)|d\mu_{2}(y)}\geq\mu_{2}^{+}(U_{x}).\end{equation} Integrating this over $M_{1}$ and using Fatou's lemma, we deduce \begin{equation}\label{firstsumminproof}\liminf_{n\to \infty}\int_{M}{|\nabla_{y} f_{n}|d\mu}\geq \int_{M_{1}}{\mu_{2}^{+}(U_{x})d\mu_{1}(x)}.\end{equation}
The first summand on the left-hand side of (\ref{whattoshow}) could be estimated analogously, but instead, we will estimate it using the assumption that $(M_{1}, \mu_{1})$ and $(M_{2}, \mu_{2})$ admit lower isoperimetric functions $J_{1}$ and $J_{2}$, respectively. First, by Fubini's formula, we have \begin{equation}\label{startsecsuminproof}
\int_{M}{|\nabla_{x}f_{n}|d\mu}=\int_{M_{1}}{\int_{M_{2}}{|\nabla_{x}f_{n}|d\mu_{2}}d\mu_{1}}\geq \int_{M_{1}}{\left|\nabla_{x}\int_{M_{2}}{f_{n}(x, y)d\mu_{2}(y)}\right|d\mu_{1}(x)}.\end{equation}
Now let us consider on $M_{1}$ the function $$F_{n}(x)=\int_{M_{2}}{f_{n}(x, y)d\mu_{2}(y)}.$$
Note that $F_{n}(x)$ is a monotone increasing sequence of non-negative smooth functions on $M_{1}$, such that \begin{equation}\label{limitofFn}F(x):=\lim_{n\to \infty}F_{n}(x)=\mu_{2}(U_{x}).\end{equation} Since $F_{n}$ is smooth for all $n$, we deduce that the sets $\{F_{n}>t\}$ have smooth boundary, so that we can apply the isoperimetric inequality on $M_{1}$, that is, $$\mu_{1}^{+}\{F_{n}>t\}\geq  J_{1}(\mu_{1}\{F_{n}>t\}).$$ Hence, we obtain, using (\ref{startsecsuminproof}) and the co-area formula, \begin{align*}\int_{M}{|\nabla_{x}f_{n}|d\mu}\geq \int_{M_{1}}{|\nabla_{x}F_{n}|d\mu_{1}}&=\int_{0}^{\infty}{\mu_{1}'\{F_{n}=t\}dt}\\&=\int_{0}^{\infty}{\mu_{1}^{+}\{F_{n}>t\}dt}\\&\geq \int_{0}^{\infty}{J_{1}(\mu_{1}\{F_{n}>t\})dt}.\end{align*}
Passing to the limit as $n\to \infty$, we get by Fatou's lemma, using the continuiuty of $J_{1}$, \begin{equation}\label{secondsuminproof}\limsup_{n\to \infty}\int_{M}{|\nabla_{x}f_{n}|d\mu}\geq \int_{0}^{\infty}{J_{1}(\mu_{1}\{F>t\})dt}.\end{equation}
By the isoperimetric inequality on $M_{2}$ with function $J_{2}$ and by (\ref{limitofFn}), $$\mu_{2}^{+}(U_{x})\geq J_{2}(\mu_{2}(U_{x}))=J_{2}(F(x)),$$ whence combining this with (\ref{firstsumminproof}) and (\ref{secondsuminproof}), we get \begin{equation}\label{combiningbothsum}\limsup_{n\to \infty}\int_{M}{|\nabla_{x} f_{n}|d\mu}+\limsup_{n\to \infty}\int_{M}{|\nabla_{y} f_{n}|d\mu}\geq \int_{M_{1}}{J_{2}(F(x))d\mu_{1}(x)}+\int_{0}^{\infty}{J_{1}(\mu_{1}\{F>t\})dt}.\end{equation}
Let us set $$\varphi(t)=\mu_{1}\{F>t\}$$ and note that $\varphi$ is monotone decreasing and right-continuous. Let $\phi$ be the generalized inverse function to $\varphi$ defined by (\ref{defgeneinv}). Then we obtain by (\ref{itselfgeneralized}), \begin{equation}\label{equimeaspsiF}\sup\{s>0:\phi(s)>t\}=\mu_{1}\{F>t\},\end{equation} which means that $\phi$ and $F$ are equimeasurable. Therefore, $$\varphi\leq \mu_{1}(M_{1}),\quad \phi\leq \mu_{2}(M_{2}),$$  and by (\ref{genehavesameint}), the definition of $\varphi$ and Fubini's formula, $$\int_{0}^{\infty} {\phi(t) dt} = \int_{0}^{\infty}{\varphi(t)dt}=\int_{M_{1}}{Fd\mu_{1}}=\mu(U)=v.$$ Hence, the pair $\varphi$, $\phi$ satisfies the condition in (\ref{intvarphipsieq}). Note that by (\ref{equimeaspsiF}), $$\int_{M_{1}}{J_{2}(F(x))d\mu_{1}(x)}=\int_{0}^{\infty}{J_{2}(\phi(t))dt},$$ whence we obtain for the right-hand side of (\ref{combiningbothsum}), $$\int_{M_{1}}{J_{2}(F(x))d\mu_{1}(x)}+\int_{0}^{\infty}{J_{1}(\mu_{1}\{F>t\})dt}=\int_{0}^{\infty}{J_{2}(\phi(t))dt}+\int_{0}^{\infty} {J_{1}(\varphi(t))dt}\geq I(v),$$ which proves (\ref{whattoshow}) and thus, finishes the proof.
\end{proof}

Let $P>0$. Given two non-negative functions $f$ on $(0, +\infty)$ and $g$ on $(0, P)$ define the function $h$ on $(0, +\infty)$ by \begin{equation}\label{defvonhinineq}h(v)=\inf_{\varphi, \phi}\left(\int_{0}^{\infty} {f(\varphi(t))dt}+\int_{0}^{\infty}{g(\phi(s))ds}.\right),\end{equation} where $\varphi$ and $\phi$ are generalized mutually inverse functions on $(0, +\infty)$ such that \begin{equation}\label{gleichheitderint}\int_{0}^{\infty} {\varphi(t) dt} = \int_{0}^{\infty}{\phi(s)ds}=v.\end{equation} and with the condition that $\phi< P$.

\begin{lemma}\label{afunctiineq}
Let $f$ and $g$ be continuous functions on the intervals $(0, +\infty)$ and $(0, P)$, respectively and suppose that $g$ is symmetric with respect to $\frac{1}{2}P$.
Also, assume that the functions $\frac{f(x)}{x}$ and $\frac{g(y)}{y}$ are monotone decreasing while the functions $f$ and $g$ are monotone increasing on the intervals $(0, +\infty)$ and $\left(0, \frac{P}{2}\right)$, respectively. Then, for any $v>0$, \begin{equation}\label{statementoffuncineq}h(v)\geq \min\left(\frac{1}{6}h_{0}(v),\frac{1}{8} f\left(\frac{v}{P}\right)P \right),\end{equation} where the function $h_{0}$ is defined for all $v>0$, by \begin{equation}\label{defvonhnull}h_{0}(v)=\inf_{\overset{xy=v}{x>0,~0<y\leq\frac{1}{2}P}}(f(x)y+g(y)x).\end{equation} 
\end{lemma}

\begin{rem}\normalfont
A similar functional inequality was stated in [\cite{grigor1985isoperimetric}, Theorem 2a] without proof.
\end{rem}

In the following we denote by $|A|$ the Lebsgue measure of a domain $A\subset \mathbb{R}^{2}$.

\begin{proof}
By an approximation argument, we can assume that $\varphi$ is strictly monotone decreasing and continuous on $(0, P)$ so that $\phi$, defined as above, is the inverse function of $\varphi$ on $(0, +\infty)$. Note that $\phi$ is then also strictly monotone decreasing and continuous and still satisfies $\phi<P$.
For such fixed $\varphi, \phi$, let us denote \begin{equation}\label{defvonI}S=\int_{0}^{\infty} {f(\varphi(t))dt}+\int_{0}^{\infty}{g(\phi(s))ds},\end{equation} so that it suffices to prove that \begin{equation}\label{whattoshowwithI}S\geq \min\left(\frac{1}{6}h_{0}(v), \frac{1}{8}f\left(\frac{v}{P}\right)P \right),\end{equation} where $h_{0}$ is defined by (\ref{defvonhnull}). 

For any $p\in (0, T)$, consider the domain $$\Phi_{p}=\{(t, s)\in \mathbb{R}^{2}:p\leq t< P, ~0\leq s \leq \varphi(t)\}$$ and for any $q>0$ the domain $$\Psi_{q}=\{(t, s)\in \mathbb{R}^{2}:s\geq q, ~0\leq t \leq \phi(s)\},$$ so that by construction, \begin{equation}\label{partitionofset}v=\int_{0}^{\infty} {\phi(s) ds}=|\Phi_{p}|+|\Psi_{q}|+pq.\end{equation} Since $\phi$ is strictly monotone decreasing and continuous, there exists $q>0$ such that $|\Psi_{q}|= \frac{1}{3}v$. Let us set $p=\phi(q)$. The proof will be split into two main cases.

\textbf{Case 1}. Let us assume that $$|\Phi_{p}|\geq \frac{1}{3}v.$$ Then we obtain by (\ref{partitionofset}), that $p\leq \frac{1}{3q}v$. By the monotonicity of $\frac{g(y)}{y}$, we therefore get \begin{align*}\int_{0}^{\infty} {g(\phi(s))ds}\geq \frac{1}{3}xg(y),\end{align*} where $x=3q$ and $y=\frac{1}{3q}v$ and similarly, \begin{align}\nonumber\int_{0}^{\infty}{f(\varphi(t))dt}\geq \frac{1}{3}f(x)y.\end{align} Hence, we obtain that $$S\geq \frac{1}{3}h_{0}(v).$$

\textbf{Case 2}. Let us now assume that $$|\Phi_{p}|< \frac{1}{3}v.$$ Then we can decrease $p$ such that $|\Phi_{p}|=\frac{1}{3}v$. Set $q=\varphi(p)$ and note that this $q$ is larger than the $q$ from Case 1, whence $$|\Psi_{q}|\leq\frac{1}{3}v,$$ so that (\ref{partitionofset}) implies \begin{equation}\label{largepq}\frac{1}{3}\leq pq\leq\frac{2}{3}v.\end{equation}

\textbf{Case 2a}. Assume further that $p\geq \frac{1}{4}P.$ It follows that $$\int_{0}^{\infty}{f(\varphi(t))dt}\geq\frac{1}{3}\frac{f(q)}{q}v$$ and since $f$ is monotone increasing, we conclude $$S\geq \frac{P}{8}f\left(\frac{v}{P}\right),$$ which proves (\ref{whattoshowwithI}).

\textbf{Case 2b}. Assume now that $p<\frac{1}{4}P$ and set $q_{0}=\varphi\left(\frac{1}{2}P\right)$.

\textbf{Case 2b(i)}. Let us first consider the case when $q_{0}\leq \frac{1}{2}q$. Using that $g(y)$ is monotone increasing on $\left(0, \frac{P}{2}\right)$, we obtain, $$\int_{0}^{\infty}{g(\psi(s))ds}\geq \frac{1}{2}g(p)q.$$ Together with $$\int_{0}^{\infty}{f(\varphi(t))dt}\geq f(q)p,$$ we deduce $$S\geq\frac{1}{2}g(p)q+f(q)p,$$ so that setting $x=\frac{v}{p}$ and $y=p$, yields \begin{align*}S\geq\frac{1}{6}\left(f(x)y+g(y)x\right)\geq h_{0}(v).\end{align*}

\textbf{Case 2b(ii)}. Finally, let us consider the case when $q_{0}>\frac{1}{2}q$. Note that the condition that $\frac{f(x)}{x}$ is monotone decreasing, implies that for any $\lambda\in (0, 1)$, $$f(\lambda x)\geq \lambda f(x).$$ Together with the monotonicity of $f$, we therefore obtain $$\int_{0}^{P/2}{f(\varphi(t))dt}\geq f(q)\frac{P}{4},$$ which yields $$S\geq f\left(\frac{v}{P}\right)\frac{P}{4},$$ and thus, proves (\ref{whattoshowwithI}) also in this case.
\end{proof}

\begin{cor}\label{thmforapp}
In the situation of Theorem \ref{thm1iso} suppose that $$\mu_{1}(M_{1})=\infty\quad\text{and}\quad\mu_{2}(M_{2})<\infty$$ and assume that $\frac{J_{1}(x)}{x}$ and $\frac{J_{2}(y)}{y}$ are monotone decreasing while the functions $J_{1}$ and $J_{2}$ are monotone increasing on the intervals $(0, +\infty)$ and $\left(0, \frac{1}{2}\mu_{2}(M_{2})\right)$, respectively. Then the manifold $(M, \mu)$ admits the lower isoperimetric function \begin{equation}\label{isoperiineqinthm2a}J(v)= c \min\left(\frac{1}{6}J_{0}(v), \frac{1}{8}J_{1}\left(\frac{v}{\mu_{2}(M_{2})}\right)\mu_{2}(M_{2}) \right),\end{equation} where function $J_{0}$ is defined for all $v>0$, by \begin{equation}\label{defvonJo}J_{0}(v)=\inf_{\overset{xy=v}{x>0,~0<y\leq\frac{1}{2}\mu_{2}(M_{2})}}(J_{1}(x)y+J_{2}(y)x),\end{equation} and the constant $c$ is defined as in Theorem \ref{thm1iso}.
\end{cor}

\begin{proof}
From Theorem \ref{thm1iso}, we know that $(M, \mu)$ has the lower isoperimetric function $cI$, where $I$ is defined by $$I(v)=\inf_{\varphi, \phi}\left(\int_{0}^{\infty} {J_{1}(\varphi(t))dt}+\int_{0}^{\infty}{J_{2}(\phi(s))ds}.\right),$$ where $\varphi$ and $\phi$ are generalized mutually inverse functions satisfying $\phi\leq \mu_{2}(M_{2})$ and the condition in (\ref{gleichheitderint}). Since $\mu_{2}(M_{2})$ is finite, we can assume that the isoperimetric function $J_{2}$ is symmetric with respect to $\frac{1}{2}\mu_{2}(M_{2})$, because the boundaries of an open set and its complement coincide in this case. Applying Lemma \ref{afunctiineq} to $I$ with $f=J_{1}$, $g=J_{2}$ and $P=\mu_{2}(M_{2})$, we obtain $$I(v)\geq \min\left(\frac{1}{6}J_{0}(v),\frac{1}{8} J_{1}\left(\frac{v}{\mu_{2}(M_{2})}\right)\mu_{2}(M_{2}) \right),$$ where function $J_{0}$ is defined by (\ref{defvonJo}), which implies that function $J$ given by (\ref{isoperiineqinthm2a}) is a lower isoperimetric function for $(M, \mu)$.
\end{proof}

\subsection{Weighted models with boundary}\label{secisobdry}

Let us also consider the topological space $M=\mathbb{R}_{+}\times \mathbb{S}^{n-1}$, $n\geq 2$, where $\mathbb{R}_{+}=[0, +\infty)$, so that any point $x\in M$ can be written in the polar form $x=(r, \theta)$ with $r\in\mathbb{R}_{+}$ and $\theta \in \mathbb{S}^{n-1}$. We equip $M$ with the Riemannian metric $ds^{2}$ that is defined in polar coordinates $(r, \theta)$ by $$ds^{2}=dr^{2}+\psi^{2}(r)d\theta^{2}$$ with $\psi(r)$ being a smooth positive function on $\mathbb{R}_{+}$ and $d\theta^{2}$ being the Riemannian metric on $\mathbb{S}^{n-1}$.
Note that $M$ with this metric becomes a manifold with boundary $$\delta M=\{(r, \theta)\in M: r=0\}$$ and we call $M$ in this case a \textit{Riemannian model with boundary}.
The Riemannian measure $\mu$ on $M$ with respect to this metric is given by \begin{equation}\label{measonmodpsi}d\mu=\psi^{n-1}(r)drd\sigma(\theta),\end{equation} where $dr$ denotes the Lebesue measure on $\mathbb{R}_{+}$ and $d\sigma$ denotes the Riemannian measure on $\mathbb{S}^{1}$. 
Let us normalize the metric $d\theta^{2}$ on $\mathbb{S}^{n-1}$ so that $\sigma(\mathbb{S}^{n-1})=1$ and define the area function $S$ on $\mathbb{R}_{+}$ by \begin{equation}\label{defofareabd}S(r)=\psi^{n-1}(r).\end{equation}

Given a smooth positive function $h$ on $M$, that only depends on the polar radius $r$, and a measure $\widetilde{\mu}$ on $M$ defined by $d\widetilde{\mu}=h^{2}d\mu$, we obtain that the weighted manifold $(M, \widetilde{\mu})$ has the area function \begin{equation}\label{areafuncontild}\widetilde{S}(r)=h^{2}(r)S(r).\end{equation} Then the weighted manifold $(M, \widetilde{\mu})$ is called \textit{weighted model} and we get that \begin{equation}\label{fubinionmode}d\widetilde{\mu}=\widetilde{S}(r)drd\sigma(\theta).\end{equation}

\begin{thm}\label{propisoformod}
Let $(M_{0}, \mu_{0})$ be a model manifold with boundary.
Assume that there exists a constant $C_{0}>0$ such that for all $r\geq 0$, \begin{equation}\label{smallpsiass}\psi_{0}(r)\leq C_{0}.\end{equation} Assume also, that \begin{equation}\label{assonareatilde}\widetilde{S_{0}}(r)\simeq \left\{ \begin{array}{lc} r^{\delta}e^{r^{\alpha}},& r\geq 1, \\1,& r<1,\end{array}\right.\end{equation} where $\delta\in \mathbb{R}$ and $\alpha\in (0, 1]$.
Then the weighted model $(M_{0}, \widetilde{\mu_{0}})$ admits the lower isoperimetric function $J$ defined by \begin{equation}\label{lowerisoonM0}J(w)=\widetilde{c} \left\{ \begin{array}{lc} \frac{w}{(\log w)^{\frac{1-\alpha}{\alpha}}},& w\geq 2, \\c'w^{\frac{n-1}{n}},& w<2,\end{array}\right.\end{equation} where $\widetilde{c}$ is a small enough constant and $c'$ is a positive constant chosen such that $J$ is continuous.
\end{thm}

\begin{proof}
Let $\nu$ be the measure on $\mathbb{R}_{+}$ defined by $d\nu(r)=\widetilde{S_{0}}(r)dr$. Then (\ref{fubinionmode}) implies that measure $\widetilde{\mu_{0}}$ has the representation $\widetilde{\mu_{0}}=\nu\times \sigma$, where $\sigma$ is the normalized Riemannian measure on the sphere $\mathbb{S}^{n-1}$. Obviously, we have by (\ref{assonareatilde}), that $$\nu(\mathbb{R}_{+})=\int_{0}^{\infty}{\widetilde{S_{0}}(r)dr}=+\infty.$$
Since $\widetilde{S_{0}}$ is a positive, continuous and non-decreasing function on $\mathbb{R}_{+}$, we obtain from [\cite{brock2012weighted}, Proposition 3.1], that $(\mathbb{R}_{+}, \nu)$ has a lower isoperimetric function $J_{\nu}(v)$ given by $$J_{\nu}(v)=\widetilde{S_{0}}(r),$$ where $v=\nu([0, r))$.
Clearly, for small $R$, we have $J_{\nu}(v)\simeq 1$.
For large enough $R$, we obtain $$v=\int_{0}^{R}{\widetilde{S_{0}}(r)dr}\simeq  R^{\delta+1-\alpha}e^{R^{\alpha}}.$$
This implies that for large $v$, $$\log v\simeq R^{\alpha}+(\delta+1-\alpha)\log R\simeq R^{\alpha},$$ and thus, $$J_{\nu}(v)=\widetilde{S_{0}}(R)\simeq R^{\delta}e^{R^{\alpha}}=R^{\alpha-1}R^{\delta+1-\alpha}e^{R^{\alpha}}\simeq\frac{v}{(\log v)^{\frac{1-\alpha}{\alpha}}},$$ which proves that \begin{equation}\label{isoforweightedhalfex}J_{\nu}(v)=c_{0} \left\{ \begin{array}{lc} \frac{v}{(\log v)^{\frac{1-\alpha}{\alpha}}},& v\geq 2, \\1,& v<2,\end{array}\right.\end{equation} is a lower isoperimetric function of $(\mathbb{R}_{+}, \nu)$ if $c_{0}>0$ is a small enough constant. Note that $J_{\nu}$ is continuous and monotone increasing on $\mathbb{R}_{+}$ and, since $\alpha\in(0, 1]$, the function $\frac{J_{\nu}(v)}{v}$ is monotone decreasing.
Let $J_{\sigma}$ be the function defined by \begin{equation}\label{lowerisoforsphere2}J_{\sigma}(v)=c_{n}\left\{ \begin{array}{lcl} v^{\frac{n-2}{n-1}},&\textnormal{if}& 0\leq v\leq \frac{1}{2}, \\(1-v)^{\frac{n-2}{n-1}},&\textnormal{if}& \frac{1}{2}<v\leq 1 ,\end{array}\right.\end{equation} and recall that $J_{\sigma}$ is a lower isoperimetric function for $(\mathbb{S}^{n-1}, \sigma)$ assuming that the constant $c_{n}>0$ is sufficiently small.
Since we assume that $\psi_{0}$ satisfies the condition in (\ref{smallpsiass}), we can apply Corollary \ref{thmforapp} and deduce that a lower isoperimetric function $J$ of $(M_{0}, \widetilde{\mu_{0}})$ is given by \begin{equation}\label{isoperimodwithbond}J(w)=c\min\left(\frac{1}{6}J_{0}(w), \frac{1}{8}J_{\nu}\left(w\right)\right),\end{equation} where $J_{0}$ is defined by $$J_{0}(w)=\inf_{\overset{uv=w}{u>0,~0<v\leq\frac{1}{2}}}\left(J_{\nu}(u)v+J_{\sigma}(v)u\right)$$ and the constant $c>0$ is defined as in Theorem \ref{thm1iso}. 

In order to estimate $J$ in this case, let us consider the function $K$, defined for all $w>0$, by \begin{equation}\label{defvonkinapp}K(w)=\frac{J(w)}{w}=\min\left(\frac{1}{6}K_{0}(w), \frac{1}{8}K_{\nu}\left(w\right)\right),\end{equation} where $K_{0}$ is given by \begin{equation}\label{Knullforsmall}K_{0}(w)=\inf_{\overset{uv=w}{u>0,~0<v\leq\frac{1}{2}}}(K_{1}(u)+K_{\sigma}(v)),\end{equation} where $K_{\nu}(u)=\frac{J_{\nu}(u)}{u}$ and $K_{\sigma}(v)=\frac{J_{\sigma}(v)}{v}$. Observe that, since $K_{\sigma}$ is monotone decreasing, $$K_{0}(w)\geq \inf_{0<v\leq \frac{1}{2}}K_{\sigma}(v)\geq K_{\sigma}\left(\frac{1}{2}\right).$$ Note that if $w\geq 2$ and $v\leq\frac{1}{2}$, then $u=\frac{w}{v}\geq 4$. Hence, we obtain that for $w\geq 2$, $$K_{0}(w)\simeq \text{const}.$$ Substituting this into (\ref{defvonkinapp}), we get, using that $K_{\nu}$ is monotone decreasing, $K(w)\simeq K_{\nu}(w)$ for $w\geq 2$, and whence \begin{equation}\label{isomixedlarge}J(w)\simeq J_{\nu}(w)\simeq \frac{w}{(\log w)^{\frac{1-\alpha}{\alpha}}},\quad w\geq 2.\end{equation}
Note that if $w\leq 2$, the infimum is attained when $u\leq 2$ and the summands in (\ref{Knullforsmall}) are comparable. 
Observe that this holds true when $$v\simeq w^{\frac{1}{2-\frac{n-2}{n-1}}},$$ so that substituting this into (\ref{Knullforsmall}), we deduce for $w\leq 2$, $$K_{0}(w)\simeq w^{-\frac{1}{n}}.$$ Hence, we
obtain that for all $w\leq 2$, $$J_{0}(w)\simeq w^{\frac{n-1}{n}},$$ and therefore by (\ref{isoperimodwithbond}), $$J(w)\simeq w^{\frac{n-1}{n}}, \quad w\leq 2.$$ Combining this with (\ref{isomixedlarge}), we conclude that the function $J(w)$ defined by (\ref{lowerisoonM0}) is a lower isoperimetric function for the weighted model $(M_{0}, \widetilde{\mu_{0}})$.
\end{proof}

\section{On-diagonal heat kernel upper bounds}\label{On-diagonal heat kernel upper bounds}

Let $(M, \mu)$ be a weighted manifold. For any open set $\Omega\subset M$, define \begin{equation}\label{selfadjbottspec}\lambda_{1}(\Omega)=\inf_{u}\frac{\int_{\Omega}{|\nabla u|^{2}d\mu}}{\int_{\Omega}{u^{2}d\mu}},\end{equation} where the infimum is taken over all nonzero Lipschitz functions $u$ compactly supported in $\Omega$.

\begin{defin}\normalfont
We say that $(M, \mu)$ satisfies a \textit{Faber-Krahn inequality} with a function $\Lambda:(0, +\infty)\to (0, +\infty)$ if, for any non-empty precompact open set $\Omega\subset M$, \begin{equation}\label{deffaberkrahn}\lambda_{1}(\Omega)\geq \Lambda(\mu(\Omega)).\end{equation}
\end{defin}

It is well-known that a Faber-Krahn inequality (\ref{deffaberkrahn}) implies certain heat kernel upper bounds of the heat kernel (see \cite{carron1996inegalites} and \cite{grigor2006heat}).

\begin{prop}[\cite{grigor2006heat}, Theorem 5.1]\label{thmfaberheatupper}
Suppose that a weighted manifold $(M, \mu)$ satisfies a Faber-Krahn inequality (\ref{deffaberkrahn}) with $\Lambda$ being a continuous and decreasing function such that \begin{equation}\label{condforfaberkr}\int_{0}^{1}{\frac{dv}{v\Lambda(v)}}<\infty.\end{equation} Then for all $t>0$, \begin{equation}\label{upperbdheatfaber}\sup_{x\in M}p_{t}(x, x)\leq \frac{4}{\gamma(t/2)},\end{equation} where the function $\gamma$ is defined by \begin{equation}\label{upperbdheatfabert}t=\int_{0}^{\gamma(t)}{\frac{dv}{v\Lambda(v)}}.\end{equation}
\end{prop}

\begin{defin}\normalfont Let $\{M_{i}\}_{i=0}^{k}$ be a finite family of non-compact Riemannian manifolds. We say that a manifold $M$ is a \textit{connected sum} of the manifolds $M_{i}$ and write \begin{equation}\label{connectedsumdef}M=\bigsqcup_{i=0}^{k}{M_{i}}\end{equation} if, for some non-empty compact set $K\subset M$ the exterior $M\setminus K$ is a disjoint union of open sets $E_{0}, \ldots, E_{k}$ such that each $E_{i}$ is isometric to $M_{i}\setminus K_{i}$ for some compact set $K_{i}\subset M_{i}$.
\end{defin}
Conversely, we have the following definition.

\begin{defin}\normalfont Let $M$ be a non-compact manifold and $K\subset M$ be a compact set with smooth boundary such that $M\setminus K$ is a disjoint union of finitely many ends $E_{0}, \ldots, E_{k}$. Then $M$ is called a \textit{manifold with ends}.
\end{defin}

\begin{rem}\normalfont 
Let $M$ be a manifold with ends $E_{0}, \ldots, E_{k}$. Considering each end $E_{i}$ as an exterior of another manifold $M_{i}$, then $M$ can be written as in (\ref{connectedsumdef}).
\end{rem}

Let $(M=\bigsqcup_{i=0}^{k}{M_{i}}, \mu)$ be a connected sum of complete non-compact weighted manifolds $(M_{i}, \mu_{i})$ and $h$ be a positive smooth function on $M$. As before, let us consider the weighted manifold $(M, \widetilde{\mu})$, where $\widetilde{\mu}$ is defined by $d\widetilde{\mu}=h^{2}d\mu$.
By restricting $h$ to the end $E_{i}=M_{i}\setminus K_{i}$ and then extending this restriction smoothly to a function $h_{i}$ on $M_{i}$, we obtain weighted manifolds $(M_{i}, \widetilde{\mu}_{i})$, where $\widetilde{\mu}_{i}$ is given by $d\widetilde{\mu}_{i}=h_{i}^{2}d\mu$. 

\begin{thm}\label{heatonsumweightfaber}
Let $(M, \widetilde{\mu})=\left(\bigsqcup_{i=0}^{k}{M_{i}}, \widetilde{\mu}\right)$ be a weighted manifold with ends where $M_{0}$ is a model manifold with boundary so that for all $r\geq0$,  $$\psi_{0}(r)\leq C_{0}$$ and $$\widetilde{S_{0}}(r)\simeq \left\{ \begin{array}{ll} r^{\delta}e^{r^{\alpha}},& r\geq1, \\1,& r<1,\end{array}\right.$$ where $0<\alpha\leq 1$, $\delta \in \mathbb{R}$ and $\widetilde{S_{0}}$ denotes the area function of a weighted model $(M_{0}, \widetilde{\mu_{0}})$. 
Assume also that all $(M_{i}, \widetilde{\mu_{i}})$, $i=1, \ldots k$, have Faber-Krahn functions $\widetilde{\Lambda_{i}}$ such that $$\widetilde{\Lambda_{i}}(v)\geq c_{i}\left\{ \begin{array}{ll} \frac{1}{(\log v)^{\frac{2-2\alpha}{\alpha}}},& v\geq2, \\v^{-\frac{2}{n}},& v<2,\end{array}\right.$$ for constants $c_{i}>0$. Then there exist constants $C>0$ and $C_{1}>0$ depending on $\alpha$ and $n$ so that the heat kernel $\widetilde{p_{t}}$ of $(M, \widetilde{\mu})$ satisfies \begin{equation}\label{upperheatweightest}\sup_{x\in M}\widetilde{p}_{t}(x, x)\leq C\left\{ \begin{array}{ll} \exp\left(-C_{1}t^{\frac{\alpha}{2-\alpha}}\right),& t\geq1, \\t^{-\frac{n}{2}},& 0<t<1.\end{array}\right.\end{equation}
\end{thm}

\begin{proof}
It follows from Theorem \ref{propisoformod}, that $(M_{0}, \widetilde{\mu_{0}})$ has the lower isoperimetric function $J$ given by (\ref{lowerisoonM0}), that is $$J(v)=\widetilde{c} \left\{ \begin{array}{lc} \frac{v}{(\log v)^{\frac{1-\alpha}{\alpha}}},& v\geq 2, \\c'v^{\frac{n-1}{n}},& v<2,\end{array}\right.$$ where $\widetilde{c}>0$ is a small enough constant and $c'$ is a positive constant chosen such that $J$ is continuous.
Since $J$ is continuous and the function $\frac{J(v)}{v}$ is non-increasing, we obtain from [\cite{Grigoryan1999}, Proposition 7.1], that $(M_{0}, \widetilde{\mu_{0}})$ admits a Faber-Krahn function $\widetilde{\Lambda}_{0}$ given by \begin{equation}\label{faberkrahnfromisoneu}\widetilde{\Lambda}_{0}(v)=\frac{1}{4}\left(\frac{J(v)}{v}\right)^{2}\simeq \left\{ \begin{array}{lc} \frac{1}{(\log v)^{\frac{2-2\alpha}{\alpha}}},& v\geq 2, \\v^{-\frac{2}{n}},& v<2.\end{array}\right.\end{equation}
We obtain from [\cite{grigor2016surgery}, Theorem 3.4] that there exist constants $c>0$ and $Q>1$ such that $(M, \widetilde{\mu})$ admits the Faber-Krahn function \begin{equation}\label{Fabrkrahnonsum}\widetilde{\Lambda}(v)=c\min_{0\leq i\leq k}\widetilde{\Lambda}_{i}(Qv).\end{equation} Hence $(M, \widetilde{\mu})$ has a Faber-Krahn function $\widetilde{\Lambda}$, satisfying \begin{equation}\label{faberkrahnfromisonew}\widetilde{\Lambda}(v)\simeq\left\{ \begin{array}{lc} \frac{1}{(\log v)^{\frac{2-2\alpha}{\alpha}}},& v\geq 2, \\v^{-\frac{2}{n}},& v<2.\end{array}\right.\end{equation} Observe that the Faber-Krahn function $\widetilde{\Lambda}$ satisfies condition (\ref{condforfaberkr}). Thus, we can apply Proposition \ref{thmfaberheatupper}, which yields the heat kernel upper bound in (\ref{upperbdheatfaber}). Hence, it remains to estimate the function $\gamma$ from the right hand side of (\ref{upperbdheatfaber}) by using (\ref{upperbdheatfabert}). In the case when $t>0$ is small enough, we get by (\ref{upperbdheatfabert}) and (\ref{faberkrahnfromisonew}), $$t=\int_{0}^{\gamma(t)}{\frac{dv}{v\widetilde{\Lambda}(v)}}=C'\int_{0}^{\gamma(t)}{\frac{dv}{v^{1-\frac{2}{n}}}}=C'\gamma(t)^{\frac{2}{n}} ,$$ which implies for some constant $C''>0$, $$\gamma(t)=C'' t^{\frac{n}{2}}.$$ For large enough $t$ on the other hand, we deduce $$t=\int_{0}^{\gamma(t)}{\frac{dv}{v\widetilde{\Lambda}(v)}}\simeq \int_{2}^{\log(\gamma(t))}{u^{\frac{2-2\alpha}{\alpha}}du}\simeq \log(\gamma(t))^{\frac{2-\alpha}{\alpha}}.$$ Therefore, $$\gamma(t)\simeq \exp\left(\textnormal{const}~t^{\frac{\alpha}{2-\alpha}}\right),$$ where $\textnormal{const}$ is a positive constant depending on $\alpha$ and $n$. Substituting these estimates for $\gamma(t)$ into (\ref{upperbdheatfaber}), we obtain the upper bound (\ref{upperheatweightest}) for the heat kernel $\widetilde{p}_{t}$ of $(M, \widetilde{\mu})$ for small and large values of $t$. For the intermediate values of $t$, we deduce the upper bound (\ref{upperheatweightest}) from the fact that the function $t\mapsto \sup_{x\in M}\widetilde{p}_{t}(x, x)$ is continuous.
\end{proof}

\begin{example}\normalfont
In Theorem \ref{heatonsumweightfaber} one can take $(M_{i}, \widetilde{\mu_{i}})=(\mathbb{H}^{n}, \mu_{i})$, $i=1, \ldots k$, where $\mu_{i}$ is the Riemannian measure on the hyperbolic space $\mathbb{H}^{n}$ since for all $0<\alpha\leq 1$, we have $$\Lambda_{\mathbb{H}^{n}}(v)\simeq\left\{ \begin{array}{lc} 1,& v\geq 2, \\v^{-\frac{2}{n}},& v<2\end{array}\right.\geq c\left\{ \begin{array}{ll} \frac{1}{(\log v)^{\frac{2-2\alpha}{\alpha}}},& v\geq2, \\v^{-\frac{2}{n}},& v<2.\end{array}\right.$$
\end{example}

\begin{rem}\normalfont
Let $(M, \widetilde{\mu})$ be the weighted manifold with ends, defined as in Theorem \ref{heatonsumweightfaber}, so that $\widetilde{S_{0}}(r)\simeq e^{r^{\alpha}}r^{\delta}$ for $r>1$ and hence, for $R>1$, $$\widetilde{V_{0}}(R)=\int_{0}^{R}{\widetilde{S_{0}}(r)dr}\simeq\int_{0}^{R}{e^{r^{\alpha}}r^{\delta}dr}\simeq e^{R^{\alpha}}R^{\delta+1-\alpha}.$$ Then, we obtain from [\cite{Coulhon1997}, Proposition 3.4] for large enough $R$, $$\widetilde{\lambda_{1}}(\Omega_{R})\leq 4\left(\frac{\widetilde{S_{0}}(R)}{\widetilde{V_{0}}(R)}\right)^{2}\leq\frac{C}{R^{2-2\alpha}},$$ where $\Omega_{R}=\{(r, \theta)\in M_{0}: 0<r<R\}$. Hence, setting $R(t)=t^{\frac{1}{2-\alpha}}$, [\cite{Coulhon1997}, Proposition 2.3] yields the following lower bound for the heat kernel $\widetilde{p_{t}}$ in $(M, \widetilde{\mu})$ for large enough $t$: 
\begin{align*}\sup_{x}\widetilde{p_{t}}(x, x)\geq\frac{1}{\widetilde{\mu}(\Omega_{R})}\exp\left(-\widetilde{\lambda_{1}}(\Omega_{R})t\right)\geq\frac{C_{1}}{e^{R^{\alpha}(t)}R^{1-\alpha}(t)}\exp\left(-\frac{Ct}{R^{2-2\alpha}(t)}\right)\geq \frac{C_{1}}{e^{C_{2}t^{\frac{\alpha}{2-\alpha}}}},\end{align*} which shows that the exponential decay in the upper bound given in (\ref{upperheatweightest}) is sharp. 
\end{rem}

\subsection{Weighted models with two ends}

Let $M$ be the topological space $M=\mathbb{R}\times \mathbb{S}^{n-1}$, $n\geq 2$, that is, any point $x\in M$ can be written in the polar form $x=(r, \theta)$ with $r\in\mathbb{R}$ and $\theta \in \mathbb{S}^{n-1}$.
For a fixed smooth positive function $\psi$ on $\mathbb{R}$ consider on $M$ the Riemannian metric $ds^{2}$ given by $$ds^{2}=dr^{2}+\psi^{2}(r)d\theta^{2},$$ where $d\theta^{2}$ is the standard Riemannian metric on $\mathbb{S}^{n-1}$. The Riemannian measure $\mu$ on $M$ with respect to this metric is given by \begin{equation}\label{muontwoendwithS}d\mu=\psi^{n-1}(r)drd\sigma(\theta),\end{equation} where $dr$ denotes the Lebesgue measure on $\mathbb{R}$ and $d\sigma$ the Riemannian measure on $\mathbb{S}^{n-1}$.
As before, we normalize the metric $d\theta^{2}$ on $\mathbb{S}^{n-1}$ so that $\sigma(\mathbb{S}^{n-1})=1$.
Then we define the area function $S$ on $\mathbb{R}$ by \begin{equation}\label{defSonttwo}S(r)=\psi^{n-1}(r).\end{equation}
Given a smooth positive function $h$ on $M$, that only depends on the polar radius $r\in \mathbb{R}$, and considering the measure $\widetilde{\mu}$ on $M$ defined by $d\widetilde{\mu}=h^{2}d\mu$, we get that the weighted model $(M, \widetilde{\mu})$, has the area function \begin{equation}\label{areaoftwoweight}\widetilde{S}(r)=h^{2}(r)S(r).\end{equation}


The Laplace-Beltrami operator $\Delta_{\mu}$ on $M$ can be represented in the polar coordinates $(r, \theta)$ as follows:
\begin{equation}\label{laplaceonmodel}\Delta_{\mu}=\frac{\partial^{2}}{\partial r^{2}}+\frac{S'(r)}{S(r)}\frac{\partial}{\partial r}+\frac{1}{\psi^{2}(r)}\Delta_{\theta},\end{equation} where $\Delta_{\theta}$ is the Laplace-Beltrami operator on $\mathbb{S}^{n-1}$. If we assume that $u$ is a radial function, that is, $u$ depends only on the polar radius $r$, we obtain from (\ref{laplaceonmodel}), that $u$ is harmonic in $M$ if and only if \begin{equation}\label{harmonicontwoendsu}u(r)=c_{1}+c_{2}\int_{r_{1}}^{r}{\frac{dt}{S(t)}},\end{equation}where $r_{1}\in [-\infty, +\infty]$ so that the integral converges and $c_{1}, c_{2}$ are arbitrary reals.

\begin{thm}\label{heatkernelforsmallendviah}
Let $(M, \mu)=(M_{0}\sqcup M_{1}, \mu)$ be a Riemannian model with two ends, where $M_{0}=\{(r, \theta)\in M: r\geq 0\}$ is a model manifold with boundary such that for all $r\geq0$, $$\psi_{0}(r)=e^{-\frac{1}{n-1}r^{\alpha}}.$$ Also assume that $(M_{1}, \mu_{1})$ is a Riemannian model with \begin{equation}\label{assumptiononM1on2end}\int_{1}^{\infty}{\frac{dt}{S_{1}(t)}}<\infty,\end{equation} and Faber-Krahn function $\Lambda_{1}$, so that \begin{equation}\label{faverkrahnohnoneside}\Lambda_{1}(v)\geq c_{1}\left\{ \begin{array}{ll} \frac{1}{(\log v)^{\frac{2-2\alpha}{\alpha}}},& v\geq2, \\v^{-\frac{2}{n}},& v<2,\end{array}\right.\end{equation} for some constant $c_{1}>0$. Then there exist positive constants $C_{x}=C_{x}(x, \alpha, n)$ and $C_{1}=C_{1}(\alpha, n)$ such that the heat kernel of $(M, \mu)$ satisfies, for all $x\in M$, the inequality \begin{equation}\label{heatkerneluppernew}p_{t}(x, x)\leq C_{x}\left\{ \begin{array}{ll} \exp\left(-C_{1}t^{\frac{\alpha}{2-\alpha}}\right),& t\geq1, \\t^{-\frac{n}{2}},& 0<t<1.\end{array}\right.
        \end{equation}
\end{thm}

\begin{proof}
Observe that the assumption (\ref{assumptiononM1on2end}) yields that we can choose positive constants $\kappa_{1}$ and $\kappa_{2}$ so that the smooth function $h$ on $M$ defined by $$h(r)=\kappa_{1}+\kappa_{2}\int_{1}^{r}{\frac{dt}{S(t)}},$$ is positive in $M$ and satisfies $h\simeq 1$ in $\{r\leq 0\}$. Consider the weighted model with two ends $(M, \widetilde{\mu})$, where $\widetilde{\mu}$ is defined by $d\widetilde{\mu}=h^{2}d\mu$. It follows from (\ref{faverkrahnohnoneside}) that the weighted model $(M_{1}, \widetilde{\mu_{1}})$ has the Faber-Krahn function $\widetilde{\Lambda}_{1}$ satisfying $$\widetilde{\Lambda}_{1}(v)\geq \widetilde{c}_{1}\left\{ \begin{array}{ll} \frac{1}{(\log v)^{\frac{2-2\alpha}{\alpha}}},& v\geq2, \\v^{-\frac{2}{n}},& v<2,\end{array}\right.$$ for some constant $\widetilde{c}_{1}>0$. Further, note that $$h|_{M_{0}}(r)\simeq \left\{ \begin{array}{ll} r^{1-\alpha}e^{r^{\alpha}},& r\geq 1, \\1,& 0\leq r<1,\end{array}\right.$$ whence the area function $\widetilde{S_{0}}$ of the weighted model with boundary $(M_{0}, \widetilde{\mu_{0}})$ admits the estimate \begin{equation}\label{estimatestilononeend}\widetilde{S_{0}}(r)\simeq \left\{ \begin{array}{ll} r^{2-2\alpha}e^{r^{\alpha}},& r\geq 1, \\1,& 0\leq r<1.\end{array}\right.\end{equation}
Since also $\psi_{0}\leq 1$, we can apply Theorem \ref{heatonsumweightfaber} and obtain that there exist constants $C>0$ and $C_{1}>0$ depending on $\alpha$ and $n$ so that the heat kernel $\widetilde{p_{t}}$ of $(M, \widetilde{\mu})$ satisfies \begin{equation}\label{largeendheatinproof}\sup_{x\in M}\widetilde{p}_{t}(x, x)\leq C\left\{ \begin{array}{ll} \exp\left(-C_{1}t^{\frac{\alpha}{2-\alpha}}\right),& t\geq1, \\t^{-\frac{n}{2}},& 0<t<1.\end{array}\right.\end{equation}
Using that $h$ is harmonic in $M$, we have by (\ref{relationofheatkernelsends}), for all $t>0$ and $x\in M$, the identity $$\widetilde{p}_{t}(x, x)=\frac{p_{t}(x, x)}{h^{2}(x)},$$ which together with (\ref{largeendheatinproof}) implies the upper bound (\ref{heatkerneluppernew}) and thus, finishes the proof.
\end{proof}

\begin{rem}\normalfont
Consider the end $\Omega:=\{r>0\}$ of the Riemannian model $(M, \mu)$ from Theorem \ref{heatkernelforsmallendviah} and note that $\left(\overline{\Omega}=\{r\geq0\}, \mu|_{\{r\geq0\}}\right)$ is parabolic by [\cite{Grigorextquotesingleyan1999}, Proposition 3.1], whence the estimate (\ref{heatkerneluppernew}) implies that we cannot get a polynomial decay of the heat kernel in $M$ as it follows from (\ref{nonparaendpoldecglob}) in Theorem \ref{thmlowerptbd}, just by assuming the polynmial volume growth condition (\ref{PolygrowthOM}).
\end{rem}

\begin{rem}\normalfont
Consider again the end $\Omega:=\{r>0\}$ of the Riemannian model $(M, \mu)$ from Theorem \ref{heatkernelforsmallendviah} and assume for simplicity that $n=2$. Let $M_{0}$ be defined as in Theorem \ref{thmlocharsph}, that is, there exists a compact set $K_{0}\subset M_{0}$ that is the closure of a non-empty open set, such that $\Omega$ is isometric to $M_{0}\setminus K_{0}$. Let us check which conditions from Theorem \ref{thmlocharsph} are not satisfied in $M_{0}$. A simple computation shows that the area function $S_{0}$ of the manifold $M_{0}$ satisfies $S_{0}''(r)\sim\alpha^{2}e^{-r^{\alpha}}r^{2\alpha-2}~\text{as}~r\to +\infty$, so that $-\frac{S_{0}''(r)}{S(r)}\to 0~\text{as}~r\to +\infty$. Together with the fact that on a compact set, the curvature is non-negative, it then follows from (\ref{gaussin2S}) that the curvature on $M_{0}$ is bounded below, which implies that $M_{0}$ is a locally Harnack manifold. Obviously, $S_{0}$ also satisfies the conditions (\ref{langsamevariquo}) and (\ref{spherhartwodimmod}) from Proposition \ref{propforspherintwo}, whence we obtain that on $M_{0}$ the spherical Harnack inequality (\ref{annuliharnack2}) holds. On the other hand, condition (\ref{upperboundVOm}) in $M_{0}$ fails, since for fixed $\rho>0$, the volume $V(x, \rho)$ can be arbitrarily small when $r\to +\infty$ where $x=(r, \theta)\in \Omega$. Hence, we have that in general, we can not drop the condition (\ref{upperboundVOm}) in Theorem \ref{thmlocharsph} to get the polynomial decay (\ref{lowerbndendparamitharnaanuohneomloc}) of the heat kernel in $M$.
\end{rem}

\bibliographystyle{abbrv}

\bibliography{librarynext}

\end{document}